\documentclass[preprint,12pt]{elsarticle}
\usepackage{lineno}
\usepackage{pgf,tikz}
\usetikzlibrary{arrows}
\usepackage{array,multirow,makecell}
\setcellgapes{1pt} \makegapedcells
\usepackage{amsmath,amssymb,amsthm}
\usepackage[colorlinks=true
linkscolor=red, urlcolor=cyan]{hyperref}
\usepackage{subcaption}
\usepackage{graphicx}
\usepackage{here}
\usepackage{epsfig}
\usepackage{color}
\usepackage{pgf,tikz}
\usetikzlibrary{arrows}
\usepackage{bbold}
\usepackage{adjustbox}
\usepackage{booktabs, longtable, multirow}
\usepackage{titlesec}
\usepackage{amsthm}
\usepackage{mathtools}
\DeclarePairedDelimiter{\ceil}{\lceil}{\rceil}
\usepackage{subcaption}

\usepackage{algorithm,algpseudocode} 
\newtheorem{theorem}{\textbf{Theorem}}
\newtheorem{lemma}{\textbf{Lemma}}

\newtheorem{definition}{\textbf{Definition}}

\newtheorem{assumptions}{\textbf{}}

\newcommand{\comment}[1]{}

\makeatletter
\newcommand\subsubsubsection{\@startsection{paragraph}{4}{\z@}{-2.5ex\@plus -1ex \@minus -.25ex}{1.25ex \@plus .25ex}{\em{\normalsize}}}
\newcommand\subsubsubsubsection{\@startsection{subparagraph}{5}{\z@}{-2.5ex\@plus -1ex \@minus -.25ex}{1.25ex \@plus .25ex}{\em{\normalsize}}}
\makeatother

\usepackage{amsmath}

\DeclareMathOperator*{\argmin}{arg\,min}

\journal{}
\begin{document}
\begin{frontmatter}

\title{Spectral coefficient learning physics informed neural network for time-dependent fractional parametric differential problems
}
\author[label1]{S M Sivalingam}
\ead{siva915544@gmail.com}
\author[label1]{V Govindaraj}
\ead{govindaraj.maths@gmail.com}
\author[label2]{A. S. Hendy\corref{cor1}}
\ead{ahmed.hendy@fsc.bu.edu.eg}
\cortext[cor1]{Corresponding author: ahmed.hendy@fsc.bu.edu.eg}
\address[label1]{Department of Mathematics, National Institute of Technology Puducherry, Karaikal-609609, India} 
\address[label2]{Department of Computational Mathematics and Computer Science, Institute of Natural Sciences and Mathematics, Ural Federal University, 19 Mira St., Yekaterinburg, Russia 620002.} 

\begin{abstract}
The study of parametric differential equations plays a crucial role in weather forecasting and epidemiological modeling. These phenomena are better represented using fractional derivatives due to their inherent memory or hereditary effects. This paper introduces a novel scientific machine learning approach for solving parametric time-fractional differential equations by combining traditional spectral methods with neural networks. Instead of relying on automatic differentiation techniques, commonly used in traditional Physics-Informed Neural Networks (PINNs), we propose a more efficient global discretization method based on Legendre polynomials. This approach eliminates the need to simulate the parametric fractional differential equations across multiple parameter values. By applying the Legendre-Galerkin weak formulation to the differential equation, we construct a loss function for training the neural network. The trial solutions are represented as linear combinations of Legendre polynomials, with the coefficients learned by the neural network. The convergence of this method is theoretically established, and the theoretical results are validated through numerical experiments on several well-known differential equations.
\end{abstract}

\end{frontmatter}


\section{Introduction}

Scientific machine learning (SciML) involves devising numerical algorithms to solve complex real-world problems using machine learning (ML) based methods. This area of research has several implications in the field of fluid dynamics \cite{Fluid dynamics}, aerospace engineering \cite{aerospace-engineering}, and control engineering \cite{control-engineering}. This field was initially popularized by the works of Lagaris et al. in \cite{Lagaris}. In the early $2000$s, several works were initiated contributing to the field of SciML. Such as the works of using evolutionary algorithms to solve ordinary differential equations (ODEs) by Cao et al. in \cite{early-1}, elliptic partial differential equations (PDEs) in \cite{early-3}, and unsupervised method to solve differential equations (DEs) in \cite{early-2}. All these studies were supported by the universal approximation capabilities of the neural networks (NN) presented as universal approximation theorem or UAT as provided by Cyberneko et al. in \cite{UAT}.

In the recent decade, the works on incorporating physics-based constraints inside the loss function of the neural network have gained high interest following the pioneering work of Raissi et al. in \cite{PINN-Raissi}. This breakthrough was named the physics-informed neural network or PINN. This has led to the revitalization of the interest in the field of SciML. This work was later extended to solve fractional order problems in \cite{FPINN-Pang} by Pang et al. Further extension of these methods to stochastic type equations was performed by Zhang et al. and Yang et al. in \cite{Zhang-PINN-Stochastic} and \cite{Yang-PINN-Stochastic}. The study of functions using NN later motivated the research community to study operators using the NN. This was achieved in \cite{Deeponet}. This introduced the NN to solve DEs by operator-based approximation, where the entire solution operator is approximated by using a two networks called the branch and trunk network.

Recently, the interest in hybridizing classical numerical and ML-based methods has gained momentum for highly challenging problems where classical methods lose accuracy. The famous works on combining the traditional numerical method with ML can be seen in the works of Mall et al. in \cite{supporting-leg-1}, where the authors employed the Legendre activation function in the neural network. Later, these were extended for faster learning of ODE by Yang et al. in \cite{supporting-leg-2}. The scheme developed was very similar to the spectral methods. These methods were based on extreme learning machines. Such numerical methods were devised for the fractional differential equations (FDEs) in \cite{supporting-leg-3}.

Similar works on combining the finite element method with ML were performed by Mitusch et al. in \cite{FEM-PINN-1}. Lienen et al. in \cite{FEM-PINN-2} introduced a spatiotemporal forecasting procedure when spare observations are available using the finite element networks (FEN). The idea was based on obtaining a PDE from the data by using the FEMs. The error analysis of such FENs was performed in \cite{FEM-PINN-3} by Kapustsin et al. Additionally, studies on using the FENs for solving the elastic-plastic problems, the incompressible fluid flow, and its associated inverse-type problems were investigated in \cite{FEM-PINN-4} and \cite{FEM-PINN-5}, respectively. These studies were extended to solve PDEs on arbitrary domains in \cite{FEM-PINN-6}. This study was based on UAT for NN. It utilized the Galerkin discretized weak PDE formulation in the loss function of the PINN framework to approximate the unknown function. Further, an application to the transient thermal conduction problem, a benchmark engineering problem, was solved using the proposed approach. 

The problems mentioned above dealt with non-parametric DEs. The study of DE with parameters following a probability distribution is more complicated to address. Methods to solve such DEs with integer order derivative were first discussed by Choi et al. in \cite{Spectral-PINN-1}. This study considered the source function of the PDE to be parametrized by a random parameter following an arbitrary distribution. They used the weak form of PDE with the Legendre function as a test function in the loss function. The solution was approximated by a linear combination of Legendre polynomials, where the coefficients were approximated by using the DNN. The complete error analysis of the method for linear type PDE was performed in \cite{Spectral-PINN-3}. Xia et al. later extended this method for PDE in the unbounded domain in \cite{Spectral-PINN-4}. Additionally, the classical FEM was coupled with the weak form, and a FEM-based operator learning method was proposed by Lee et al. in \cite{Spectral-PINN-5}. Here, the authors used the source function at arbitrary parameters as the input to the neural network. The spectral method with PINN was extended to solve problems where the initial or the boundary condition, source function, or the coefficients in the PDE were randomly chosen from an arbitrary distribution in \cite{Spectral-PINN-7}. Here, the authors demonstrated the method's performance on several well-known PDEs, such as the Burgers equation, Navier Stokes equation, and advection-diffusion equation. 

Though several numerical methods to simulate parametric DE with integer order derivative exist in the literature, the study of parametric DE with fractional order derivative is unexplored in the literature. Further, it is seen that there is no classical numerical method to simulate the parametric FDEs. Thus, simulation at each parameter value is required to obtain the solution. This is highly time-consuming and less efficient than simulating non-parametric FDEs. Therefore, SciML offers an innovative approach to solving such parametric FDEs with less computation cost and less memory. 

From all the above discussion and the need for a numerical method to simulate the parametric FDEs, we propose a novel numerical method for parametric FDEs by combining the traditional Legendre-Galerkin scheme and the PINN as an extension of the method proposed for integer order systems in \cite{Spectral-PINN-1}. Further, we present a detailed theoretical analysis along with the convergence of the proposed method. Thus, the overall contributions of this work are as follows:
\begin{itemize}
\item A novel numerical method to simulate the parametric FDEs is proposed.
\item The convergence of the proposed method is discussed under certain assumptions on the non-linear operator. 
\item Further, the extension of the proposed method to multi-dimensional problems is also discussed.
\item The theoretical convergence is proved numerically by applying the proposed method to several well-known linear and non-linear FDE problems.
\end{itemize}

The rest of the manuscript is organized as follows: Section \ref{Preliminaries} presents the essential fundamentals required for formulating the problem and the methodology. Section \ref{proposed-method} formulates the method for a one-dimensional FDE along with the construction of the solution and the formulation of the network. Section \ref{error-analysis-of-the-proposed-method} derives the error of the proposed method and presents a discussion on the convergence of the proposed method along with the theoretical proofs. Section \ref{Extension of the Proposed method to multidimensional FDE} extends the proposed method to multi-dimensional FDEs. The proposed method is validated on numerical examples, and the theoretical proofs are validated in Section \ref{Numerical experiments} followed by the discussion and conclusions in \ref{conclusion}.

\section{Preliminaries}\label{Preliminaries}
This section presents the preliminary concepts, such as the definition of the fractional derivative and the structure of the Legendre polynomial. All results are presented by considering $z\in \mathbb{R}$, which depends on either one variable or multiple-variable, usually called univariate or multivariate functions. Further we denote the fractional order to be $\zeta$ such that $\zeta \in (\kappa-1, \kappa],~\kappa \in \mathbb{N}$. The symbol $z^{(n)}$ is used to denote the $n^{th}$ order derivative of the function $z$.

Further we define $W^{\tilde{q},\hat{r}}(\Omega)$as the fractional order Sobolev space of functions \cite{fractional-sobolev-spaces} on $\Omega$ equipped with the norm $\|\cdot\|_{W^{\tilde{q},\hat{r}}(\Omega)}$, where $\tilde{q}\geq 0$ and $1\leq \hat{r} < \infty$. Particularly, $L^2(\Omega) = W^{0,2}(\Omega)$ and $H^{\tilde{q}}(\Omega) = W^{\tilde{q},2}(\Omega)$. We consider $\Omega = [0,1]$ for simplicity.

\begin{definition}\cite{1} \label{Caputo-der-definition}
The $\zeta^{th}$ fractional derivative in the sense of Caputo for a univariate function $z$ is defined by
\begin{equation}\nonumber
\centering
\begin{array}{lcl}
^C_0D_x^{\zeta} z(x) & = & \left\lbrace 
\begin{array}{ll} 
\displaystyle \frac{1}{\Gamma(\kappa-\zeta)}\int^{x}_0 (x-s)^{\kappa-\zeta-1} z^{(\kappa)}(s) ds,& \displaystyle  \zeta\in (\kappa-1,\kappa),\\
\displaystyle \frac{d^{\zeta}z(x)}{dx^{\zeta}},& \displaystyle \kappa=\zeta.
\end{array}
\right.
\end{array}
\end{equation} 
Accordingly, for a multivariate function $z :\Omega^n \to \mathbb{R}$, can be expressed as follows:
\begin{equation}\nonumber
\centering
\begin{array}{lcl}
^C_0D_{x_1}^{\zeta} z(x_1,x_2,\cdots,x_n) & = & \left\lbrace 
\begin{array}{ll} 
\displaystyle \frac{1}{\Gamma(\kappa-\zeta)}\int^{x_1}_0 (x_1-s)^{\kappa-\zeta-1} z^{(\kappa)}(s,x_2,\cdots,x_n) ds,& \displaystyle  \zeta\in (\kappa-1,\kappa),\\
\displaystyle \frac{d^{\zeta}z(x_1,x_2,\cdots,x_n)}{dx_1^{\zeta}},& \displaystyle \kappa=\zeta.
\end{array}
\right.
\end{array}
\end{equation} 
\end{definition}
\begin{definition}\cite{1}
The fractional integral corresponding to the derivative given in Definition \ref{Caputo-der-definition} for the univariate function $z$ is given by  
\begin{equation}\nonumber
\centering
^C_0D_x^{\zeta} z(x) = \frac{1}{\Gamma(\zeta)}\int^{x}_0 (x-s)^{\zeta-1} z(s) ds,~x>0.
\end{equation}
Similarly, the fractional derivative of a multivariate function $z:\Omega^n \to \mathbb{R}$ is given by
\begin{equation}\nonumber
\centering
^C_0D_{x_1}^{\zeta} z(x_1,x_2,\cdots,x_n) = \frac{1}{\Gamma(\zeta)}\int^{x_1}_0 (x_1-s)^{\zeta-1} z(s,x_2,\cdots,x_n) ds,~x_1>0.
\end{equation}
\end{definition}
\begin{definition}\cite{Spectral-PINN-3}
The Rademacher complexity of the $\mathbb{F}$, for a family of $\{\tilde{\Upsilon}_i\}_{i=1}^L$ of i.i.d random variables distributed according to distribution $\mathbb{P}_{\Omega}$ is given by 
\begin{equation} \nonumber
\centering
R_L(\mathbb{F}) = \mathbb{E}_{\{\tilde{\Upsilon}_i,\epsilon_i\}_{i=1}^L }\left[ \sup_{f \in \mathbb{F}}\left| \frac{1}{L} \sum_{i=1}^L \epsilon_i f(\tilde{\Upsilon}_i) \right| \right],
\end{equation}
where $\epsilon_i$'s are Bernoulli random variables.
\end{definition}
\subsection{Legendre polynomials}
The Legendre polynomials denoted by $\tilde{P}_n(x)$ on an interval $x\in[-1,1]$ are mutually orthogonal polynomials with respect to the $L^2$ inner product. These are defined by the recurrence relation
\begin{equation} \nonumber
\centering
\begin{array}{lcl}
(n+1)\tilde{P}_{n+1}(x)&=&(2n+1)x\tilde{P}_n(x)-n\tilde{P}_{n-1}(x),\\
\tilde{P}_0(x) &=&1,\\
\tilde{P}_1(x)&=&x.
\end{array}
\end{equation}
We use the shifted Legendre polynomial since our study involves the domain $[0, X],~X>0$. For this we shift the domain of $x$ by using the transformation $\displaystyle\frac{2x}{X}-1$ as in \cite{siva-1}. The following series formula defines the shifted Legendre polynomial:
\begin{equation} \label{Legendre-polynomial-series}
\centering
\begin{array}{lcl}
\hat{P}_n(x)&=&\displaystyle  \sum_{k=0}^n(-1)^{n+k}\frac{(n+k)!x^k}{X^k(n-k)!(k!)^2}.
\end{array}
\end{equation}
To impose boundary conditions implicitly, we follow the approach in \cite{fractional-spectral-source}, we define the modified Legendre polynomial as follows:
\begin{equation}\label{modified-Legendre-polynomial}
\centering
P_n(x) = \hat{P}_n(x)+a_n\hat{P}_{n+1}(x)+b_n\hat{P}_{n+2}(x),
\end{equation}
where the constants $a_n$ and $b_n$ are determined according to the initial condition.

\section{The proposed neural network-based method} \label{proposed-method}
This section presents the neural network-based approach to solving parametric time-FDE. Initially, a linear problem is considered, and the neural network-based method is derived, and then its theoretical analysis is presented. Later, a discussion on the extension to the non-linear case is presented. We use $f$ to denote the function when it is non-parametric and $f(x;\Upsilon)$ to denote the parametric function with parameter $\Upsilon$.
\subsection{Nonparametric one-dimensional FODEs}
For clarification purposes, we initiate by considering a nonparametric one-dimensional FODEs
\begin{equation} \label{1d-problem-general-form}
\centering
\begin{array}{lcl}
^C_0D_x^{\zeta} z + \hat{v} z^{(1)} &=& f,
\end{array}
\end{equation}
with the homogeneous Dirichlet boundary conditions
\begin{equation} \nonumber
\centering
z(0) = 0,~ z(1)=0,
\end{equation}
or with the homogeneous Neumann conditions
\begin{equation} \nonumber
\centering
z^{(1)}(0) = 0,~z^{(1)}(1) = 0.
\end{equation}
Any solution $z(x)$ to the problem \eqref{1d-problem-general-form} satisfies the following residual:
\begin{equation} \label{residual-for-PINN}
\centering
\begin{array}{lcl}
R(z)&=&\displaystyle(^C_0D_x^{\zeta} z + \hat{v} z^{(1)} - f).
\end{array}
\end{equation}
The above residual \eqref{residual-for-PINN} is utilized for training the network in the PINN procedure. This helps to obtain what is called the strong solution of the DE. However, not all DEs hold for a strong solution. Further calculation of higher-order derivatives when using strong form is difficult. Thus, a variational or weak form of the equation \eqref{1d-problem-general-form} is used to approximate the solution. The variational form of the equation \eqref{1d-problem-general-form} is given by
\begin{equation} \label{weak-form}
\centering
\begin{array}{lcl}
\displaystyle\int_0^1 \left( ^C_0D_x^{\zeta} z v(x)+ \hat{v} z v^{(1)}(x) \right) dx &=& \displaystyle\int_0^1 f(x)v(x)dx,~v\in H,
\end{array}
\end{equation}
where $H$ is some suitable Hilbert space. Before discussing the equation, it becomes essential to discuss the existence and uniqueness of the solution of such a system. As the system is linear in nature it becomes obvious that a solution exists for the system. The basic idea of using the variational form is to approximate the solution using some series approximation. Using the Legendre-Galerkin method \cite{fractional-spectral-source}, the solution of the equation \eqref{1d-problem-general-form} can be approximated as 
\begin{equation} \label{trial-solution}
\centering
z_N(x) = \sum_{k=1}^{N-1} \omega_k P_k(x),
\end{equation} 
where $N$ is the total number of basis functions used and $P_k$ denotes the $k^{th}$ order Legendre polynomial. Thus, the problem reduces to determine the value of $\omega_k$, which satisfies the following equation:
\begin{equation} \nonumber
\centering
\begin{array}{lcl}
\displaystyle\int_0^1 \left( ^C_0D_x^{\zeta} z_N v_N(x)+ \hat{v} z_N v_N^{(1)}(x) \right) dx &=& \displaystyle\int_0^1 f(x)v_N(x)dx,~v_N\in H_N,
\end{array}
\end{equation} 
where $H_N=\text{span}\{P_1, P_2, \cdots, P_{N-1}\}$. This can be further expressed in the following form:
\begin{equation} \nonumber
\centering
\displaystyle \left(\hat{H}+\hat{v}\hat{M}\right)\hat{\omega} = \hat{F},
\end{equation}
where $\displaystyle\hat{H}= (h)_{ij}$, $\displaystyle\hat{M}=(m)_{ij}$, $\displaystyle\hat{\omega} = \{\omega_1,\cdots,\omega_{N-1}\} \in \mathbb{R}^{N-1}$ and $\displaystyle\hat{F} = (f)_{ij}$, $\displaystyle h_{ij} =\int_{0}^1 P^{\zeta}_i(x)P_j(x)dx$, $\displaystyle m_{ij}=\int_{0}^1 P^{(1)}_i(x)P_j(x)dx$, $\displaystyle f_{ij}=\int_{0}^1 f(x)P_j(x)dx$ and 
\begin{equation} \label{Legendre-polynomial-derivative}
\centering
\begin{array}{lcl}
\hat{P}^{(1)}_n(x)&=&\displaystyle  \sum_{k=1}^n(-1)^{n+k}\frac{k(n+k)!x^{k-1}}{X^k(n-k)!(k!)^2}
\end{array}
\end{equation}
and
\begin{equation} \label{Legendre-polynomial-derivative-fractional}
\centering
\begin{array}{lcl}
\hat{P}^{\zeta}_n(x)&=&\displaystyle  \sum_{k=\ceil{\zeta}}^n(-1)^{n+k}\frac{\Gamma(k+1)(n+k)!x^{k-\zeta}}{X^k(n-k)!\Gamma(1-\zeta+k)(k!)^2}.
\end{array}
\end{equation}
The above algebraic equation can be solved using any numerical methods available in the literature, such as matrix inversion or elimination. Thus, the solution of the system can be obtained as
\begin{equation} \nonumber
\centering
\hat{\omega}  =\left(\hat{H} + \hat{v}\hat{M}\right)^{-1} F.
\end{equation}
The above technique is similarly constructed as in the spectral methods available in literature \cite{zayernouri2024spectral} for FDEs.
\subsection{Parametric one-dimensional FODEs}
If the function $f$ is parametrized by a random number $\Upsilon$. This $\Upsilon$ is assumed to be obtained from a probability space $(\hat{\Omega}, \mathbb{F}, \mathbb{P})$, where the domain $\Omega$ is assumed to be compact. Accordingly, we have
\begin{equation} \label{1d-problem-parametric-form}
\centering
\begin{array}{lcl}
^C_0D_x^{\zeta} z + \hat{v} z^{(1)} &=& f(x;\Upsilon).
\end{array}
\end{equation}
Further, as the parameter $\Upsilon$ is random, training using normal spectral methods is impossible since it requires obtaining $\omega_k's$ for different random numbers $\Upsilon \in \hat{\Omega}$. Thus, modern machine learning procedures are incorporated to learn the coefficients $\omega_k's$ for different values of $\Upsilon$. 
Thus, the solution representation as from \cite{Spectral-PINN-1} and \cite{Spectral-PINN-3} is of the form
\begin{equation} \label{trial-solution-for-parametric}
\centering
z(x;\Upsilon) = \sum_{k=1}^{N-1} \hat{\omega}_k(\Upsilon) P_k(x).
\end{equation}
To train this network at various values of $\Upsilon's$, the following form of loss function is utilized:
\begin{equation} \nonumber
\centering
\begin{array}{lcl}
\mathfrak{L}(\omega)&=&\displaystyle \mathbb{E}_{\Upsilon \sim \mathbb{P}_{\hat{\Omega}}}\left(\sum_{k=1}^{N-1}\left(\int_0^1 \left( ^C_0D_x^{\zeta} z(x; \Upsilon) P_k(x)+ \hat{v} z(x; \Upsilon) P_k^{(1)}(x) \right) dx\right.\right.\\
&&\left.\left.- \displaystyle\int_0^1 f(x; \Upsilon)P_k(x)dx\right)^2\right)\\
&=& \displaystyle \mathbb{E}_{\Upsilon \sim \mathbb{P}_{\hat{\Omega}}}\left(\sum_{k=1}^{N-1}\left( \left(\left(\hat{H}+\hat{v}\hat{M}\right)\hat{\omega}(\Upsilon)\right)_k - \left(\hat{F}(\Upsilon)\right)_k\right)^2\right).
\end{array}
\end{equation}
The subscript $k$ indicates the evaluation of the loss function at $P_k(x)$. In simple words, it means choosing the $k^{th}$ row of the matrices. This can be further simplified as follows:
\begin{equation} \nonumber
\centering
\begin{array}{lcl}
\mathfrak{L}(\omega)&=&\displaystyle \int_{\hat{\Omega}} \sum_{k=1}^{N-1}\left|\left(\left(\hat{H}+\hat{v}\hat{M}\right)\hat{\omega}(\Upsilon)\right)_k - \left(\hat{F}(\Upsilon)\right)_k\right|^2d\Upsilon\\
&=&\displaystyle \left\|\left(\hat{H}+\hat{v}\hat{M}\right)\hat{\omega}(\Upsilon) - \hat{F}(\Upsilon)\right\|^2_{L^2(\hat{\Omega})}.
\end{array}
\end{equation}
The above loss function $\mathfrak{L}(\omega)$ by using using Monte-Carlo integration \cite{Monte-Carlo-formula} simplifies as follows:
\begin{equation} \label{discrete-residual-minimization-problem}
\centering
\begin{array}{lcl}
\mathfrak{L}_R(\omega)&=&\displaystyle \frac{|\hat{\Omega}|}{L}\sum_{m=1}^{L}\left(\sum_{k=1}^{N-1}\left(\int_0^1 \left( ^C_0D_x^{\zeta} z(x; \Upsilon) P_k(x)+ \hat{v} z(x; \Upsilon) P_k^{(1)}(x) \right) dx\right.\right.\\
&&\left.\left.- \displaystyle\int_0^1 f(x; \Upsilon)P_k(x)dx\right)^2\right)\\
&=& \displaystyle \frac{|\hat{\Omega}|}{L}\sum_{m=1}^{L}\left|\left(\hat{H}+\hat{v}\hat{M}\right)\hat{\omega}(\Upsilon) - \hat{F}(\Upsilon)\right|^2.
\end{array}
\end{equation}
\subsection{Parametric Nonlinear one-dimensional FODEs}
 The proposed method is suitable for all forms of linear problems, and further, it can also be extended to non-linear problems. For this, the following form of non-linear system is considered:
\begin{equation}\label{1d-problem-general-form-nonlinear}
\centering
^C_0D_x^{\zeta} z + \mathcal{N}(z) = f(x;\Upsilon),
\end{equation}
where $\mathcal{N}(z)$ is nonlinear in terms of $z$. Similar to the previous one, the weak form of the FDE \eqref{1d-problem-general-form-nonlinear} can be obtained by multiplying $v\in H$ and integrating over $\Omega$. This yields:
\begin{equation}\nonumber
\centering
\begin{array}{lcl}
\displaystyle\int_0^1 \left( ^C_0D_x^{\zeta} z v(x)+ \mathcal{N}(z) v(x) \right) dx &=& \displaystyle\int_0^1 f(x)v(x)dx,~v\in H.
\end{array}
\end{equation}
Further, integration by parts can be performed to remove the higher order derivatives of the function $z$ similar to the one in \cite{Spectral-PINN-7}. Based on the assumption of existence of solution to the nonlinear FDE, a trial solution \eqref{trial-solution} is used in the weak form. This yields
\begin{equation} \nonumber
\centering
\begin{array}{lcl}
\displaystyle\int_0^1 \left( ^C_0D_x^{\zeta} z_N v_N(x)+ \mathcal{N}(z_N) v_N(x) \right) dx &=& \displaystyle\int_0^1 f(x)v_N(x)dx,~v_N\in H_N.
\end{array}
\end{equation} 
Further, through simplification, one can obtain a non-linear system of the form
\begin{equation}\nonumber
\centering
\begin{array}{lcl}
\hat{A}(\hat{\omega}) = \hat{F},
\end{array}
\end{equation}
where the function $\hat{A}$ depends on the nonlinear operator $\mathcal{N}$. Considering the system \eqref{1d-problem-general-form}, the function $\hat{A}$ is a matrix $\left(\hat{H} + v\hat{M}\right)(\hat{\omega})$. While the function $f$ is parametrized by $\Upsilon$, the nonlinear system takes the form
\begin{equation}\nonumber
\centering
\hat{A}(\hat{\omega}(\Upsilon)) = \hat{F}.
\end{equation}
Further when $\hat{F}$ is parametrized by using $\Upsilon$ the solution in the form of \eqref{trial-solution-for-parametric} can be obtained by minimizing
\begin{equation}\nonumber
\centering
\begin{array}{lcl}
\mathfrak{L}(\omega)&=&\displaystyle \int_{\hat{\Omega}} \sum_{k=1}^{N-1}\left|\left(\hat{A}\left(\hat{\omega}(\Upsilon)\right)\right)_k - \left(\hat{F}(\Upsilon)\right)_k\right|^2d\Upsilon,\\
&=&\displaystyle \left\|\hat{A}\left(\hat{\omega}(\Upsilon)\right) - \hat{F}(\Upsilon)\right\|^2_{L^2(\hat{\Omega})}
\end{array}
\end{equation} 
for $\hat{\omega}$. This is a residual minimization problem as in \eqref{discrete-residual-minimization-problem}.
Now, it is to be noted that for the problem \eqref{1d-problem-general-form}, the matrix $\hat{H}$ and $\hat{M}$ are not symmetric. Thus, the proposed method generally does not construct a symmetric matrix $\hat{A}$. Further, for non-linear problem \eqref{1d-problem-general-form-nonlinear}, symmetricity is not possible. For the linear case, this is shown by generating the matrix $\hat{H}$ and $\hat{M}$ for $N=12$ in the Appendix. Also, it is observed that the matrix $\hat{H}$ is not any special kind of matrix. Further, $\hat{H}$ and $\hat{M}$ are nonsingular and thus invertible. Also, the determinant of their inverse is non-zero. To proceed further, the following assumptions are made on the FDE:
\begin{assumptions} \label{assump:H1}
For $\hat{\omega}_1,~\hat{\omega}_2 \in C([0,T])$, we have
\begin{equation}\nonumber
\centering
\|\hat{\omega}_1-\hat{\omega}_2\|_2^2 \leq C_{\hat{A}}(\|\hat{\omega}_1\|_2^2,\|\hat{\omega}_2^2\|) \|\hat{A}(\hat{\omega}_1)-\hat{A}(\hat{\omega}_2)\|_2^2,
\end{equation}
where $C_{\hat{A}}>0$ is a constant which depends on both $\hat{\omega}_1$, and $\hat{\omega}_2$ and also $\hat{A}(\hat{\omega})$ in norm is bounded and non-zero.
\end{assumptions}
The above assumption can be viewed as a bound on the linearization of the inverse of $\hat{A}$, that is, if $\bar{A}$ is the linearisation of $\hat{A}$, then $\|\bar{A}^{-1}\|_2^2\leq C_{\hat{A}}$. In addition, the bound $C_{\hat{A}}$ is directly influenced by $\mathcal{N}$. The above assumption is true whenever the inverse of the linearization of $\mathcal{N}$ is bounded. This is true when $\mathcal{N}$ takes the form $\Delta z$, $z\Delta z$, $z^2$, $z^3$ and $\sin(z)$.  Further, such assumptions have been utilized to provide an error estimate for PINN in \cite{PINN-error-estimates}, for elliptic PDE in \cite{elliptic PDE} and as explained in \cite{weak-PINN} the above bound is necessary for the stability of the differential equation which ensures that the total can be estimated by residuals. It also suggested for nonlinear DEs the stability often relies on the regularity of the unknown solution of the DE.

\begin{assumptions}\label{assump:H2}
For bounded constants $a_n,~b_n$ and $c_n$, the modified Legendre polynomial ${P}_n(x)$ for $x \in [0,T]$, $T<\infty$ is bounded and let its bound be given by $C_{{P}}$, that is 
\begin{equation}\nonumber
\centering
\sup_{x\in [0,T], n\in \{0,\cdots,N-1\}}\{\left|P_n(x)\right|\} \leq C_{{P}}.
\end{equation}
Also, let the derivatives ${P}_n^{(i)}(x)$ and ${P}_n^{\zeta}(x)$ satisfy:
\begin{equation}\nonumber
\centering
\begin{array}{lcl}
\displaystyle\sup_{x\in [0,T], n\in \{0,\cdots,N-1\}}\{\left|{P}_n^{(i)}(x)\right|\}&\leq& C_{{P}^{(i)}},\\
\displaystyle\sup_{x\in [0,T], n\in \{0,\cdots,N-1\}}\{\left|{P}_n^{\zeta}(x)\right|\}&\leq& C_{{P}^{\zeta}}.\\
\end{array}
\end{equation}
\end{assumptions}
The Assumption \ref{assump:H2} holds for the modified Legendre polynomial since the shifted Legendre polynomial $\hat{P}_n(x)$ is bounded on a bounded domain $\Omega$ and its derivatives ($\hat{P}_n^{(i)}(x)$ and $\hat{P}_n^{\zeta}(x)$) are also bounded. This directly makes the modified Legendre polynomial satisfy the Assumption \ref{assump:H2}.

\begin{assumptions}\label{assump:H3}
Let $\mathcal{N}$ be Lipschitz, then for $\hat{\omega}_1,~\hat{\omega}_2 \in C([0,T])$, $\mathcal{A}$ satisfy the following condition:
\begin{equation}\nonumber
\centering
\|\hat{A}(\hat{\omega}_1)-\hat{A}(\hat{\omega}_2)\|_2^2 \leq C_{\hat{A}_L}\|\hat{\omega}_1-\hat{\omega}_2\|_2^2,
\end{equation}
where $C_{\hat{A}_L}>0$ and depends on Lipschitz constant of the function $\mathcal{N}$ denoted by $L$ and the bounds $C_P$, $C_P^{(i)}$ and $C_P^{\zeta}$. 
\end{assumptions}
The Assumption \ref{assump:H3} is inherently satisfied when $\mathcal{N}$ is linear. For non-linear problems, it is satisfied if the $\mathcal{N}$ is Lipschitz. For example, if $\mathcal{N}(\cdot)$ is given by $\sin(z)$ or $\exp(z)$ or any other Lipschitz function.
Further, to summarize, the Assumptions \ref{assump:H2} are inherently satisfied if $n$ (polynomial order), $x$ (domain), and $i$ (derivative order) are finite. The Assumptions \ref{assump:H1} and \ref{assump:H3} are inherently satisfied for linear functions. For non-linear functions, these are satisfied when $\hat{A}$ is bi-Lipschitz. That is when $\hat{A}$ is injective Lipschitz and its inverse given by $\hat{A}^{-1}$ or the inverse of linearization of $\hat{A}$ given by $\bar{A}^{-1}$ is also Lipschitz. Such $N(z)$ are $z^3$ when $z$ is non-zero and $\sin(z)$ when $z$ is bounded. 
\subsection{Formulation of neural network}
This section describes the neural network architecture utilized in this study. The proposed numerical method blends the Legendre-Galerkin spectral method and the deep neural network architecture. This tries to increase the generalization ability of the Legendre spectral method. As explained in the previous section, the Legendre-Galerkin spectral method is a basic approach to obtain the solutions of FDEs. However, this approach is not suitable for parametrized problems. So, to overcome this, a deep neural network is used to learn the coefficient $\omega_k's$ for each value of $\Upsilon$. This is done by training the neural network to approximate the solution of the DEs $z$ at various values of $\Upsilon$. One advantage of this approach is obtaining a closed-form solution of the DE for all the values of $\Upsilon$. Also, it is suitable for very large dimensional problems when $\Upsilon$ has a huge dimension. To study this, the neural network architecture is first described. 

The neural network used is a basic DNN with $H$ number of layers. Let $\omega_N(x)$ represent the neural network architecture and its mathematical description is given by
\begin{equation} \nonumber
\centering
\begin{array}{lcl}
\omega_N^{r}(x) &=&\left\lbrace\displaystyle
\begin{array}{ll}
\mathcal{W}^{1}x+\mathcal{B}^{1},&r=1,\\
\mathcal{W}^{r}\sigma(\omega_N^{r-1})+\mathcal{B}^{r},&2\leq l \leq H,
\end{array}
\right.
\end{array}
\end{equation}
where $\mathcal{W}^{r}\in\mathbb{R}^{n_r \times n_{r-1}}$ and $\mathcal{B}^{r}\in\mathbb{R}^{n_r}$ denotes the weights and the bias used in the $r^{th}$ hidden layer, $\sigma$ denotes the activation function used. This study uses the $\tanh$ activation function since it is commonly used for approximating the solutions to DEs. Further let $\hat{n} = (n_0, \cdots, n_L)$. Let the family of network parameters be denoted by $\hat{\theta} = \hat{\theta}_{\hat{n}}=\{ (\mathcal{W}^{1},\mathcal{B}^{(1)}),\cdots, (\mathcal{W}^{H},\mathcal{B}^{H}) \}$ and its realization as a neural network by $\mathcal{R}[\hat{\theta}](x)$. For a given $\hat{n}$, the collection of all possible parameters be denoted by
\begin{equation}\nonumber
\centering
\hat{\Theta}_{\hat{n}} = \{\{(\mathcal{W}^{r},\mathcal{B}^{r})\}:\mathcal{W}^{r}\in\mathbb{R}^{n_r \times n_{r-1}}~\text{and}~\mathcal{B}^{r}\in\mathbb{R}^{n_r}\}.
\end{equation}
For two different neural network architectures, we write $\hat{n}_1 = \{n_1^{(1)},\cdots,n_L^{(1)}\} \subset \hat{n}_2 = \{n_1^{(2)},\cdots,n_L^{(2)}\}$ if $\forall$ $\hat{\theta}_1 \in \hat{\Theta}_{\hat{n}_1}$, $\exists$ $\hat{\theta}_2 \in \hat{\Theta}_{\hat{n}_2}$ $\ni$ $\mathcal{R}[\hat{\theta}_1](x) = \mathcal{R}[\hat{\theta}_2](x)$. Let $\{\hat{n}_n\}_{n\geq 1}$ be a sequence of networks with the inclusion $\hat{n}_n \subset \hat{n}_{n+1}$ $\forall n \in \mathbb{N}$. The corresponding sequence of the neural network is represented by
\begin{equation}\nonumber
\centering
\mathfrak{N}^*_n = \{\mathcal{R}[\hat{\theta}](x), \hat{\theta} \in \hat{\Theta}_{\hat{n}_n}\}.
\end{equation}
Since the function $f$ parametrized by $\Upsilon$ is continuous and bounded, the unknown coefficients $\omega^*$ are also continuous and bounded and let the bound be given by $\sup_{\Upsilon\in \hat{\Omega}}\left|\omega^*(\Upsilon)\right| = C_{\omega^*}$. Also, the activation function $\sigma$ is bounded and let its upper and lower bound be given by $\sigma^{-}$ and $\sigma^{+}$. Let us define $h:[\sigma^{-},\sigma^{+}] \to [-C_{\omega^*},C_{\omega^*}]$ which is affine continuous function satisfying $h(\sigma^{-}) = -C_{\omega^*}$ and $h(\sigma^{+}) = C_{\omega^*}$. Thus, the task is to approximate the target function $\omega^*$. Let 
\begin{equation}
\centering
\displaystyle\mathfrak{N}_n = \left\lbrace h (\sigma(g_n)),~g_n\in \mathfrak{N}^*_n\right\rbrace.
\end{equation}
Then by the definition of $h$, $\displaystyle\|g\|_{C({\hat{\Omega}})}\leq C_{\omega^*},~\forall g\in \cup_{n\in \mathbb{N}} \mathfrak{N}_n$. Thus we have $\displaystyle\sigma^{-1}(h^{-1}(\omega^*))$ to be well defined and continuous on a compact set $\hat{\Omega}$. Then, by the UAT the following holds
\begin{equation}
\centering
\begin{array}{lcl}
\displaystyle\lim_{n\to\infty} \inf _{\hat{g}\in\mathfrak{N}_n^*} \left\| \hat{g}-\sigma^{-1}(h^{-1}(\omega^*)) \right\|_{C({\hat{\Omega}})} = 0.
\end{array}
\end{equation}
Additionally using the continuity of $h(\sigma(\cdot))$ we have
\begin{equation}
\centering
\begin{array}{lcl}
\displaystyle \lim_{n\to\infty} \inf _{g \in\mathfrak{N}_n} \left\| g-\omega^* \right\|_{C({\hat{\Omega}})} = \lim_{n\to\infty} \inf _{\hat{g}\in\mathfrak{N}_n^*} \left\| h(\sigma(\hat{g}))-\omega^* \right\|_{C({\hat{\Omega}})}= 0.
\end{array}
\end{equation}
This shows that $\exists$ $g \in \mathfrak{N}_n$ such that it could approximate the unknown coefficients $\omega^*$ for sufficiently large $n$.
\subsection{Construction of solution}
The solution of the DE $z(x; \Upsilon)$ is approximated using a linear combination of Legendre polynomials. This is given by 
\begin{equation}\nonumber
\centering
\tilde{z}_N(x; \Upsilon) = \sum_{k=1}^{N-1}\omega^*_k(\Upsilon) P_k(x),
\end{equation}
where $P_k$ is the modified Legendre polynomial which satisfies the boundary condition inherently as explained in Section \ref{Preliminaries} and $\omega^*_k \in C(\hat{\Omega}, \mathbb{R}^{N-1})$ are the set of coefficients to be determined. Further $*$ is used to denote the coefficients that satisfy the weak form \eqref{weak-form}. The $\tilde{z}(x)$ presented above is similar to the Legendre-Galerkin approximation. Now $\tilde{z}$ depends on the random number $\Upsilon$. Thus for $\tilde{z}$ to be the solution of parametric equation \eqref{1d-problem-general-form} or \eqref{1d-problem-general-form-nonlinear}, the following should be satisfied:
\begin{equation}\nonumber
\centering
\omega_k^* = \argmin_{\omega \in C(\hat{\Omega}, \mathbb{R}^{N-1})} \mathfrak{L}(\omega).
\end{equation}
The coefficients $\omega^*$ are now approximated using the DNN given by $\omega_N$. Thus, the modified objective is given by
\begin{equation}\nonumber
\centering
\hat{\omega_n} = \argmin_{\omega_N \in \mathfrak{N}_n} \mathfrak{L}(\omega_N).
\end{equation}
Further, since the continuous form of loss function $\mathfrak{L}$ is approximated by using the discrete form given by $\mathfrak{L}_R$. The task becomes the minimization of the discrete residual minimization problem given by
\begin{equation}\nonumber
\centering
\hat{\omega_{n,L}} = \argmin_{\omega_N \in \mathfrak{N}_n} \mathfrak{L}_R(\omega_N).
\end{equation}
Thus, the corresponding solution can be represented as follows:
\begin{equation}\nonumber
\centering
\tilde{z}_{N,n,L}(x; \Upsilon) = \sum_{k=1}^{N-1}
{\hat{\omega}_{n,L,k}}
(\Upsilon) P_k(x).
\end{equation}
\section{Error analysis of the proposed method}\label{error-analysis-of-the-proposed-method}
The error in the proposed method is given by 
\begin{equation}\nonumber
\centering
\|z-\tilde{z}_{N,n,L}\|_{L^2(\hat{\Omega};H^{1}(\Omega))} \to 0~\text{as}~N,n,L \to \infty.
\end{equation}
Further, by using triangle inequality, we obtain the following
\begin{equation}\nonumber
\centering
\|z-\tilde{z}_{N,n,L}\|_{L^2(\hat{\Omega})} \leq \|z-\tilde{z}_{N}\|_{L^2(\hat{\Omega})} + \|\tilde{z}_{N}-\tilde{z}_{N,n}\|_{L^2(\hat{\Omega})}+ \|\tilde{z}_{N,n}-\tilde{z}_{N,n,L}\|_{L^2(\hat{\Omega})}.
\end{equation}
The error $\|z-\tilde{z}_{N}\|_{L^2(\hat{\Omega})},$ represents the error in the approximation of the weak solution of the FDE using the Legendre-Galerkin method. It has been proved in \cite{fractional-spectral-source}, that the approximation $\tilde{z}_{N}$ approaches $z$ as the value of $N$ tends to $\infty$. However, it has been observed in \cite{siva-1, siva-2, supporting-leg-3} that as the value of $N$ increases, the solution obtained by the Legendre approximation deteriorates. Thus optimal value for $N$ has been used in \cite{siva-1, siva-2, supporting-leg-3}. This was also observed for integer order problems in \cite{Spectral-PINN-3}. 

The second part $\|\tilde{z}_{N}-\tilde{z}_{N,n}\|_{L^2(\hat{\Omega})}$ denotes the error between the target function $\omega^*$ and the $\hat{\omega_n}$. This is also called the approximation error. The following theorem proves the convergence of $\hat{\omega_n}$ to $\omega^*$, which is based on the universal approximation theorem as in \cite{UAT}.
\begin{theorem}\label{2st-convergence-theorem}
For $f\in C(\hat{\Omega},L^1(\Omega))$ and $\omega^*\in C(\hat{\Omega},\mathbb{R}^{N-1})$ and $\hat{\omega_{n}}\in \mathfrak{N}_{n}$, the error  
\begin{equation} \nonumber
\centering
\|\omega^*-\hat{\omega_n}\|_{L^2(\hat{\Omega})} \to 0~\text{as}~n\to\infty.
\end{equation}
\end{theorem}
\begin{proof}
Using the Assumption \ref{assump:H1}, the following can be obtained:
\begin{equation}\nonumber
\centering
\begin{array}{lcl}
\displaystyle\|\omega^*-\hat{\omega_n}\|_2^2& \leq&\displaystyle C_{\hat{A}}\|\hat{A}(\omega^*)-\hat{A}(\hat{\omega_n})\|_2^2\\
&\leq& \displaystyle C_{\hat{A}}\left(\|\hat{A}(\omega^*)-\hat{F}\|_2^2+ \|\hat{A}(\hat{\omega_n})-\hat{F}\|_2^2\right)\\
&\leq& \displaystyle C_{\hat{A}} \mathfrak{L}(\hat{\omega_n}).
\end{array}
\end{equation}
Similarly using Assumption \ref{assump:H3}, the following is obtained:
\begin{equation} \nonumber
\centering
\begin{array}{lcl}
\mathfrak{L}(\hat{\omega_n})&\leq& \displaystyle \inf_{\omega_N \in \mathfrak{N}_n} \mathfrak{L}(\omega_N)\\
&\leq& \displaystyle \inf_{\omega_N \in \mathfrak{N}_n} \left\lbrace\|\hat{A}(\hat{\omega_n})-\hat{A}(\omega^*)\|^2_2+\|\hat{A}(\omega^*)-\hat{F}\|_2^2\right\rbrace\\
&\leq& \displaystyle C_{\hat{A}_L} \inf_{\omega_N \in \mathfrak{N}_n} \left\lbrace\|\hat{\omega_n}-\omega^*\|^2_2\right\rbrace.
\end{array}
\end{equation}
Thus, using UAT, it can be concluded that $\displaystyle \inf_{\omega\in\mathfrak{N}_n} \|\omega^*-\hat{\omega_n}\|^2_2 \to 0$ as $n \to \infty$ and so $\mathfrak{L}(\hat{\omega_n})\to 0$ as $n\to\infty$.
\end{proof}

 The third part $\|\tilde{z}_{N,n}-\tilde{z}_{N,n, L}\|_{L^2(\hat{\omega})}$ is the generalization error of the neural network on the unseen data. Thus, now the work is to focus on establishing the generalization error that is
\begin{equation} \nonumber
\centering
\|\hat{\omega_n}-\hat{\omega_{n,L}}\|_{L^2(\hat{\omega})}.
\end{equation}

\begin{theorem}\cite{ref-1}
Let $\mathbb{F}$ be a $\tilde{P}$-uniformly bounded class of functions and $L \in\mathbb{N}$. Then for any small number $\Lambda >0$
\begin{equation}\nonumber
\centering
\sup_{f\in \mathbb{F}} \left| \frac{1}{L}\sum_{i=1}^L f(\tilde{\Upsilon}_i)-\mathbb{E}(f(\hat{\Upsilon})) \right| \leq 2R_L(\mathbb{F})+\Lambda
\end{equation}
with probability of $\displaystyle 1-\exp\left(\frac{-L\Lambda^2}{2\tilde{P}^2}\right)$.
\end{theorem}
Let $\mathbb{F}_n = \left\lbrace |\hat{A}(\omega)-\hat{F}|^2:\omega \in \mathfrak{N}_n\right\rbrace$. Further by considering $f$ to be a $\Lambda$-uniformly bounded and from the Assumptions \ref{assump:H1} and \ref{assump:H3} it could be concluded that
\begin{equation}\nonumber
\centering
\|\hat{A}(\omega)-\hat{F}\|_{L^{\infty}} \leq \tilde{P},~\text{for}~\tilde{P}>0.
\end{equation}
Thus $\mathbb{F}_n$ is $\tilde{P}$-uniformly bounded function. 
\begin{lemma} \label{supporting-lemma}
Let $\left\lbrace \Upsilon_m \right\rbrace_{m=1}^L$ be i.i.d samples from the distribution $\mathbb{P}$. Then for small $\Lambda>0$, with probability of at least $\displaystyle 1-2\exp\left( -\frac{L\Lambda^2}{32\tilde{P}^2}\right)$
\begin{equation}\nonumber
\centering
\displaystyle \sup_{\omega \in \mathbb{N}_n}\left| \mathfrak{L}_R(\omega) - \mathfrak{L}(\omega)\right| \leq 2R_n\left(\mathbb{F}_n\right) + \frac{\Lambda}{2}.
\end{equation}
\end{lemma}
\begin{theorem}\label{3st-convergence-theorem}
Let $\mathbb{F}_n= \{\hat{A}(\hat{\omega})-\hat{F}\}$, $f(x; \Upsilon)\in C(\hat{\Omega},L^1(\Omega))$ and $\displaystyle\lim_{L\to\infty} R_L(\mathbb{F}_n) = 0$ $\forall n\in \mathbb{N}$. Then with probability $1$
\begin{equation} \nonumber
\centering
\|\hat{\omega_n}-\hat{\omega_{n,L}}\|_{L^2(\hat{\Omega})} \to 0~\text{as}~n,L\to \infty.
\end{equation}
\end{theorem}
\begin{proof}
Using the Assumption \ref{assump:H1} the following is obtained:
\begin{equation}\nonumber
\centering
\begin{array}{lcl}
\displaystyle\|\hat{\omega_n}-\hat{\omega_{n,L}}\|_2^2 &\leq & \displaystyle C_{\hat{A}}\|\hat{A}(\hat{\omega_n})-\hat{A}(\hat{\omega_{n,L}})\|_2^2 \\
&\leq& \displaystyle  C_{\hat{A}} \left( \|\hat{A}(\hat{\omega_n})-\hat{F}\|^2_2+ \|\hat{A}(\hat{\omega_{n,L}})-\hat{F}\|_2^2 \right)\\
&\leq& \displaystyle \mathfrak{L}\left(\hat{\omega_n}\right) + \mathfrak{L}\left(\hat{\omega_{n,L}}\right).\\
\end{array}
\end{equation}
Now applying the Lemma \ref{supporting-lemma} with $\Lambda = 2L^{\frac{1}{2}+\epsilon}$ for $\displaystyle 0<\epsilon<\frac{1}{2}$, the following is obtained:
\begin{equation}\nonumber
\centering
\begin{array}{lcl}
\mathfrak{L}\left(\hat{\omega_{n,L}}\right) &\leq&\displaystyle \mathfrak{L}_R\left(\hat{\omega_{n,L}}\right) + 2R_L(\mathbb{F}_n) + L^{-\frac{1}{2}+\epsilon}\\
&\leq&\displaystyle  \mathfrak{L}_R\left(\hat{\omega_{n}}\right) + 2R_L(\mathbb{F}_n) + L^{-\frac{1}{2}+\epsilon}\\
&\leq&\displaystyle  \mathfrak{L}\left(\hat{\omega_{n}}\right) + 4R_L(\mathbb{F}_n) + 2L^{-\frac{1}{2}+\epsilon}.
\end{array}
\end{equation}
This results in
\begin{equation}\nonumber
\centering
\begin{array}{lcl}
\displaystyle\|\hat{\omega_n}-\hat{\omega_{n,L}}\|_2^2 &\leq & \displaystyle 2\mathfrak{L}\left(\hat{\omega_{n}}\right) + 4R_L(\mathbb{F}_n) + 2L^{-\frac{1}{2}+\epsilon}.
\end{array}
\end{equation}
Now as $L$ and $n$ tends to $\infty$ and using the Assumption \ref{assump:H3}
\begin{equation} \nonumber
\centering
\begin{array}{lcl}
\displaystyle\lim_{L\to\infty, n \to \infty}\|\hat{\omega_n}-\hat{\omega_{n,L}}\|_2^2 &\leq & \displaystyle \lim_{L\to\infty, n \to \infty} \|2\mathfrak{L}\left(\hat{\omega_{n}}\|_2^2\right) \\
&\leq&\displaystyle \lim_{L\to\infty, n \to \infty}\|\hat{\omega_n}-\omega^*\|_2^2.
\end{array}
\end{equation}
Thus it could be concluded that $\|\hat{\omega_n}-\hat{\omega_{n,L}}\|_{L^2(\Omega)} \to 0~\text{as}~n,L\to \infty$.
\end{proof}

\begin{theorem}\label{Convergence-final-theorem}
Let $f(x; \Upsilon)\in C(\hat{\Omega}, L^1(\Omega))$ and $\forall n \in \mathbb{N}$, $\displaystyle\lim_{L \to \infty}   R_L(\mathbb{F}_n) = 0$. Then, the error in the approximation 
\begin{equation} \nonumber
\centering
\lim_{n\to \infty, L\to \infty}\|z-\tilde{z}_{N,n,L}\|_{L^2(\Omega)} = 0
\end{equation}
with probability $1$. 
\end{theorem}
\begin{proof}
By using the solution in \eqref{trial-solution-for-parametric} and the Assumptions \ref{assump:H2} the error in the approximation is
\begin{equation} \nonumber
\centering
\begin{array}{lcl}
\displaystyle \int_{\Omega}\int_{I}\left|z-\tilde{z}_{N,n,L}\right|^2 dx d\Upsilon&=& \displaystyle \int_{\hat{\Omega}}\int_{I}\left|\sum_{k=1}^{N-1}\left(\omega_k^*-\hat{\omega_{n,L,k}}\right)P_k(x)\right|^2 dx d\Upsilon\\
&\leq&\displaystyle C_P\left\|\left(\omega^*-\hat{\omega_{n,L}}\right)\right\|^2_{L^2(\hat{\Omega})}.
\end{array}
\end{equation}
From the Theorem \ref{2st-convergence-theorem} and \ref{3st-convergence-theorem},  it can be concluded that $\left\|\left(\omega^*-\hat{\omega_{n,L}}\right)\right\|^2_{L^2(\hat{\Omega})}$ tends to $0$ as $n\to \infty$ and $L\to \infty$. Thus $\|z-\tilde{z}_{N,n,L}\|_{L^2(\hat{\Omega})}\to 0$ as $n\to \infty$ and $L\to \infty$.
\end{proof}
\section{Extension of the proposed method to multidimensional FDEs}\label{Extension of the Proposed method to multidimensional FDE}
This section presents the extension of the proposed method to the multidimensional function. That is when $z$ depends on many independent variables. The problem which we consider in this section is of the form
\begin{equation}\label{2d-problem-general-form-nonlinear}
\centering
\begin{array}{lcl}
_0^CD^{\zeta}_{x_1} z(x_1,\cdots,x_n) + N(z) = f(x_1,\cdots,x_n; \Upsilon),
\end{array}
\end{equation}
with 
\begin{equation}\nonumber
\centering
z = 0~\text{in}~\partial \Omega^n,
\end{equation}
where the $x_1,\cdots,x_n$ are the independent variables where $x_1$ denotes the time and the other independent variable denote the space, $N$ is a non-linear operator which is chosen such that it satisfies the Assumptions \ref{assump:H1}, \ref{assump:H2} and \ref{assump:H3} and $z$, $f$ are chosen from appropriate spaces. The weak formulation of the equation \eqref{2d-problem-general-form-nonlinear} for $\zeta \in (0,1)$ and $z\in H(\Omega^n)$, where $H$ is some suitable space such as fractional Sobolev space can be obtained as
\begin{equation}\label{weak-form-for-FPDE}
\centering
\begin{array}{lcl}
\displaystyle \int_{\Omega^n}\left( _0^CD^{\zeta}_{x_1} z(x_1,\cdots,x_n; \Upsilon) + N(z)\right) v dx_1\cdots dx_n = \displaystyle \int_{\Omega^n}  f(x_1,\cdots,x_n; \Upsilon) v dx_1\cdots dx_n,~\forall v\in H.
\end{array}
\end{equation}
The Legendre-Galerkin approximation for this system is given by
\begin{equation} \nonumber
\centering
z_N(x_1,\cdots,x_n;\Upsilon) = \sum_{k_{x_1}=1}^{N-1}\cdots\sum_{k_{x_n}=1}^{N-1} \omega_{k_{x_1},\cdots,k_{x_n}}(\Upsilon) P_{k_{x_1}}(x_1)\cdots P_{k_{x_n}}(x_n),
\end{equation}
where $\omega's$ are the coefficients obtained from the neural network. This could be used in the weak form \eqref{weak-form-for-FPDE} to obtain the loss function denoted by $\hat{\mathfrak{L}}$. Subsequently, the targeted coefficients $\omega^*$ can be obtained as
\begin{equation}\nonumber
\centering
\omega^* = \argmin_{\omega\in C(\Omega^n)}\hat{\mathfrak{L}}(\hat{\omega}).
\end{equation} 
This continuous residual minimization problem can be converted to an equivalent form of a discrete residual minimization problem as 
\begin{equation}\nonumber
\centering
\omega^*_n = \argmin_{\omega\in \hat{\mathfrak{N}}_n}\hat{\mathfrak{L}}_R(\hat{\omega}),
\end{equation}
where $\hat{\mathfrak{N}}_n$ contains all the neural networks whose output is of dimension $(N-1)^2$ with $n$ network parameters and $\hat{\mathfrak{L}}_R$ is the discrete loss function.
An example of the FDE satisfying the Assumptions \ref{assump:H1}, \ref{assump:H2} and \ref{assump:H3} is given by
\begin{equation} \nonumber
\centering
\begin{array}{lcl}
_0^CD^{\zeta}_t z(x_1,\cdots,x_n) - \hat{v} \Delta z +  z \nabla z  = f(x_1,\cdots,x_n),~\text{for}~z\in H^{\zeta}(\Omega^{n}),
\end{array}
\end{equation}
which is the non-linear advection-diffusion equation. The convergence of the proposed method for the multidimensional FDE holds under the Assumptions \ref{assump:H1}, \ref{assump:H2}, and \ref{assump:H3}. Further, the proof is similar to the univariate FDE. 
\section{Numerical experiments}\label{Numerical experiments}
This section validates the theoretical results on numerical examples. For this, we have considered linear and non-linear problems. The function $f$ on the right-hand side is constructed such that the solution $z$ is known for obtaining the numerical error in our approach. To obtain the numerical error as explained in theory, we have used the $ L^2$ error denoted by $L^2_{Te}$. We have also used $L^{\infty}$-error denoted by $L^{\infty}_{Te}$. The $Te$ in the suffix is to denote that the error is calculated in testing. Further, we present the convergence plots and the results obtained at particular random parameter values. For all the networks we have used a neural network with one hidden layer with $tanh$ activation function. 

\subsection{Linear problem}
This section presents the results obtained from a linear FDE. For this, we consider the FDE of the form
\begin{equation}\label{linear-1-d-example-eq}
\centering
\begin{array}{lcl}
_0^CD_x^{\zeta} z(x) + z^{(1)}(x) &=& f(x; \Upsilon), x \in [0,1],\\
u(0)&=&0,\\
u(1)&=&0.
\end{array}
\end{equation}
The function $f$ is given by
\begin{equation}\nonumber 
\centering
\begin{array}{lcl}
\displaystyle f(x; \Upsilon) = -\frac{m_1 m_2 \Gamma (m_2) (\zeta+(m_2+1) (x-1)) x^{m_2-\zeta}}{\Gamma (-\zeta+m_2+2)}+m_1 m_2 (1-x) x^{m_2-1}-m_1 x^{m_2},
\end{array}
\end{equation}
where $\Upsilon$ is used to represent $m_1$ and $m_2$ which are random numbers generated from the $U[3,5]$ and $x$ represents time. Thus the exact solution is $z(x;\Upsilon) = m_1 (1-x) x^{m_2}$. The fractional order $\zeta$ is chosen to be $1.5$. Initially, we obtain the quadrature nodes and weights by choosing the degree to be $10$. The value of $P$ considered to be $4$, $8$, $16$, $32$ and $64$. Similarly the value of $L$ is considered to be at $10$, $50$, $100$, $500$, and $1000$. The solution is obtained at all these values of $L$ and $P$, and the error in the approximation is obtained. The $L^2_{Te}$ and $L^{\infty}_{Te}$ error obtained on varying the number of neurons ($n$) in each layer and the number of sampling points ($L$) is given in Table \ref{Linear-example-1-error-table}. The graphical results of the convergence of the error as the number of neurons and the number of sampling points increases is demonstrated in Figure \ref{linear-1-d-example-plots}. These graphs indicate that upon increasing the number of sampling points $(L)$, the approximation by using the proposed method gets much closer to the unknown function, and the error gets reduced.
\begin{table}[H]
\centering
\begin{tabular}{|c|c|ccccc|}
\hline
\multirow{2}{*}{$L$}&$n$&\multirow{2}{*}{4}&\multirow{2}{*}{8}&\multirow{2}{*}{16}&\multirow{2}{*}{32}&\multirow{2}{*}{64}\\
\cline{2-2}
&Error&&&&&\\
\hline
\multirow{2}{*}{10}
&$L^2_{Te}$&$3.48e-03$&$9.25e-03$&$3.59e-03$&$4.34e-03$
&$3.75e-03$\\
&$L^{\infty}_{Te}$&$1.59e-02$&$7.98e-02$&$1.91e-02$&$2.68e-02$&$2.54e-02$\\
\hline
\multirow{2}{*}{50}
&$L^2_{Te}$&$3.55e-03$&$1.65e-03$&$1.24e-03$&$1.52e-03$&$1.53e-03$\\
&$L^{\infty}_{Te}$&$2.28e-02$&$1.03e-02$&$1.18e-02$&$1.30e-02$&
$1.13e-02$\\
\hline
\multirow{2}{*}{100}
&$L^2_{Te}$&$3.41e-03$&$9.05e-04$&$6.14e-04$&$6.24e-04$&$6.75e-04$\\
&$L^{\infty}_{Te}$&$1.60e-02$&$6.26e-03$&$2.82e-03$&$3.66e-03$&$3.71e-03$\\
\hline
\multirow{2}{*}{500}
&$L^2_{Te}$&$4.73e-04$&$4.16e-04$&$2.80e-04$&$1.15e-04$&$3.12e-04$\\
&$L^{\infty}_{Te}$&$2.16e-03$&$2.13e-03$&$2.01e-03$&$2.44e-03$&$2.64e-03$\\
\hline
\multirow{2}{*}{1000}
&$L^2_{Te}$&$2.70e-03$&$5.62e-04$&$2.37e-04$&$3.83e-04$&$4.38e-04$\\
&$L^{\infty}_{Te}$&$1.06e-02$&$3.83e-03$&$1.91e-03$&$2.13e-03$&$3.33e-03$\\
\hline
\end{tabular}
\caption{$L^{2}_{Te}$ and $L^{\infty}_{Te}$ error obtained as the number of parameters in each layer ($n_p$) and the number of points ($L$) from the domain $\hat{\Omega}$ is increased. \label{Linear-example-1-error-table}}
\end{table}

\begin{figure}[H]
     \begin{subfigure}[b]{0.5\textwidth}
     \includegraphics[width=1\textwidth, height=0.7\textwidth]{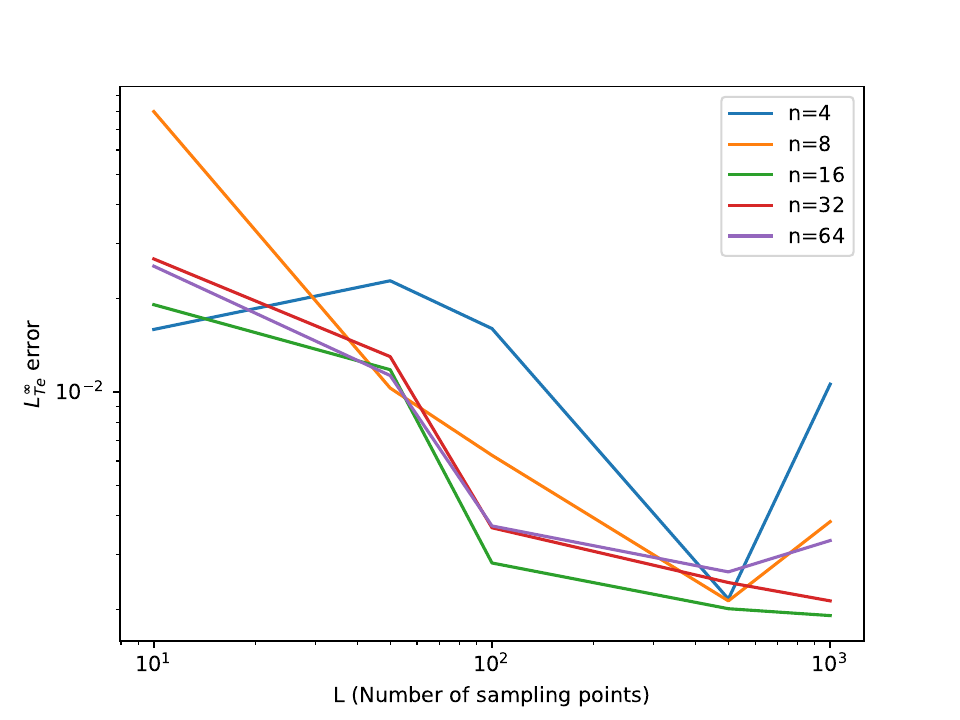}
     \caption{Convergence of $L^{\infty}_{Te}$ error as number of sampling points $L$ from the domain $\hat{\Omega}$ increases at different number of neurons $n$. \label{L_inf_hat_Omega_L}}
	 \end{subfigure}
	 \begin{subfigure}[b]{0.5\textwidth}
     \includegraphics[width=1\textwidth, height=0.7\textwidth]{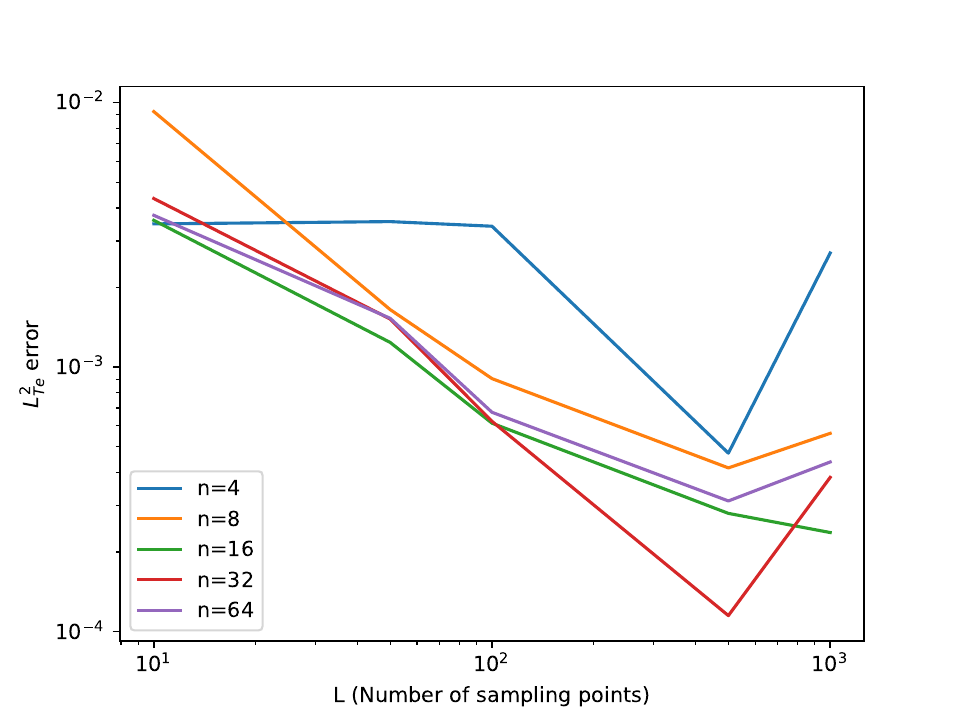}
     \caption{Convergence of $L^{2}_{Te}$ error as number of sampling points $L$ from the domain $\hat{\Omega}$ increases at different number of neurons $n$. \label{L_2_hat_Omega_L}}
	 \end{subfigure}
	 \caption{Numerical results on convergence for the parametric linear problem \eqref{linear-1-d-example-eq}.\label{linear-1-d-example-plots}}
\end{figure}

\subsection{Heat equation}
This section presents the application of the proposed method to the parametrized time-fractional heat equation with time and space denoted by $x_1$ and $x_2$ respectively. For this, we consider the following equation:
\begin{equation}\label{Heat equation}
\centering
\begin{array}{lcl}
_0^CD^{\zeta}_{x_1}z(x_1,x_2) &=& \displaystyle \hat{v} z_{x_1,x_1} + f(x_1,x_2; \Upsilon),~(x_1,x_2) \in [0,T]\times[0,X],
\end{array}
\end{equation}
where the function $f$ is chosen to be 
\begin{equation}\nonumber
\centering
\begin{array}{lcl}
f(x_1,x_2)&=&\displaystyle m_1 x^{m_2}c \left(-\frac{x_2 x_1^{-\zeta} \Gamma (m_2+1) (X-x_2) (T (\zeta-m_2-1)+m_2 x_1+x)}{\Gamma (-\zeta+m_2+2)}+2 T-2 x_1\right).
\end{array}
\end{equation}
The exact solution of this system is given by 
\begin{equation}\nonumber
\centering
u(x_1,x_2) =  m_1 c (T-x_1) (X-x_2) x_2 x_1^{m_2}.
\end{equation}
The random parameters $m_1$ and $m_2$ are assumed to follow $U[5,7]$, $X=1$, $T=1$ and $C=5$. The fractional order is chosen to be $\zeta=0.7$. We use the Gauss-Legendre quadrature of degree 10 to approximate the integral in the loss function. Further the number of sampling points $L$ is chosen to be $10$, $100$, $200$ and $300$. The weak form of the equation is formulated, and the network is trained. The results obtained by using the proposed approach are presented in Figure \ref{Heat equation-example-plots}. The plots are obtained by fixing the number of neurons in each layer to be $4$. The convergence of the solution as the number of sampling points is increased is given in Figure \ref{L_2_hat_Omega_L_heat} by using the $L^{2}_{Te}$ error and in Figure \ref{L_inf_hat_Omega_L_heat} by using the $L^{\infty}_{Te}$ error. It could be seen that the error reduces as the number of sampling points $L$ increases from $10$ to $300$. This verifies the derived theoretical results. The approximated solution of the FDE at a particular value of $m_1$ and $m_2$ is presented in Figure \ref{particular_t_heat}. These results show that the proposed method has well approximated the error to the order of $10^{-3}$. Further, the solution obtained in training and testing along with the exact solution and its error are presented in Figure \ref{Train_plot_activation_heat} and \ref{Test_plot_activation_heat}. These plots infer that the error obtained is of order $10^{-3}$ for this problem. For this problem, we trained the network for $500$ epochs. We used the Adam optimizer for the first $300$ epoch, and the rest were performed using the LBFGS optimizer.
\begin{figure}[H]
\begin{subfigure}[b]{0.45\textwidth}
\includegraphics[width=0.9\textwidth, height=0.7\textwidth]{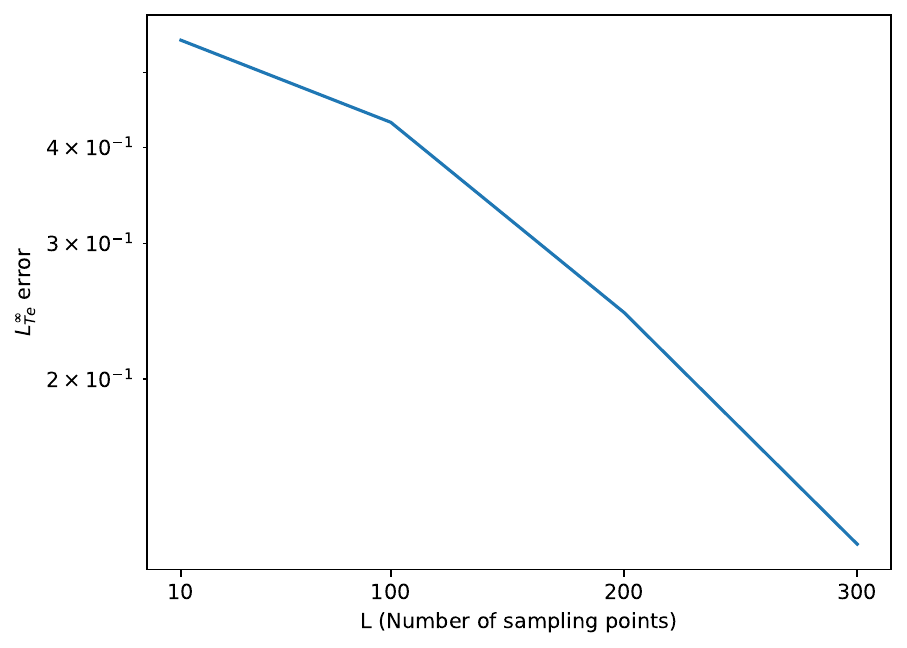}
     \caption{Convergence of $L^{\infty}_{Te}$ error as number of sampling points $L$ from the domain $\hat{\Omega}$ increases. \label{L_inf_hat_Omega_L_heat}}
	 \end{subfigure}
	 \begin{subfigure}[b]{0.45\textwidth}
     \includegraphics[width=0.9\textwidth, height=0.7\textwidth]{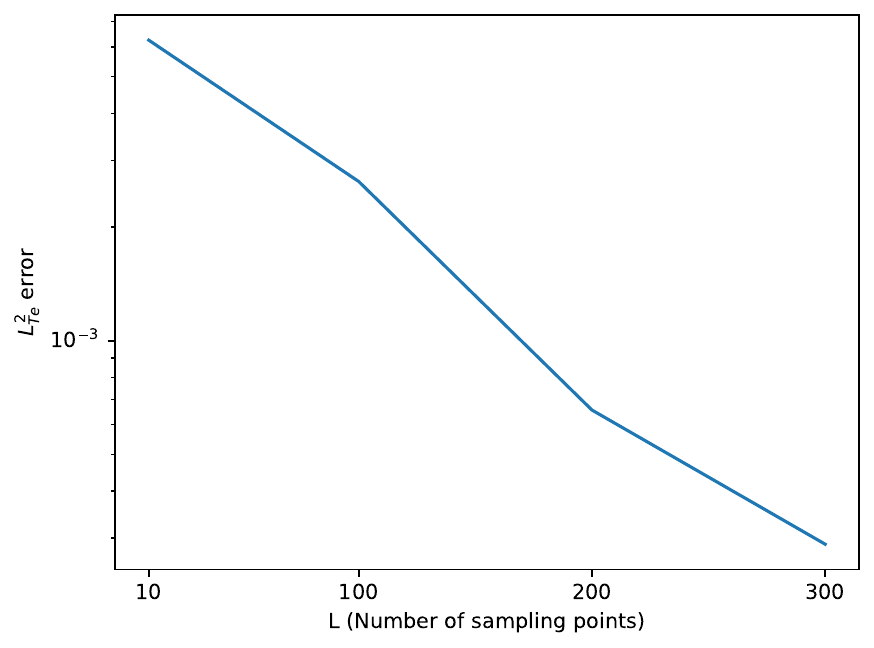}
     \caption{Convergence of $L^{2}_{Te}$ error as number of sampling points $L$ from the domain $\hat{\Omega}$ increases. \label{L_2_hat_Omega_L_heat}}
	 \end{subfigure}
	 \begin{subfigure}[b]{1\textwidth}
     \includegraphics[width=1\textwidth, height=0.3\textwidth]{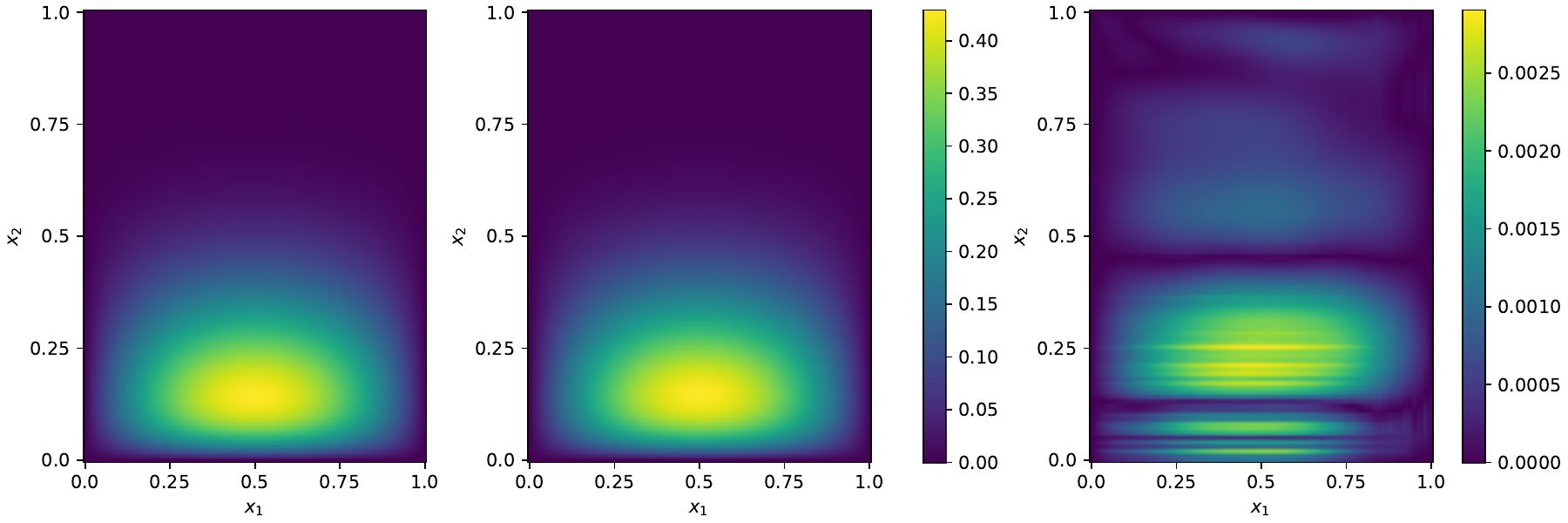}
     \caption{Approximate solution (left), exact solution (center) and absolute error (right) for a particular pair of $m_1$ and $m_2$ during training.\label{Train_plot_activation_heat}}
	 \end{subfigure}
	 \caption{Numerical results obtained for the parametric heat equation given in equation \eqref{Heat equation} for $X=1$, $T=1$, and $C=5$.\label{Heat equation-example-plots}}
\end{figure}

\begin{figure}[H]
	 \begin{subfigure}[b]{1\textwidth}
     \includegraphics[width=1\textwidth, height=0.6\textwidth]{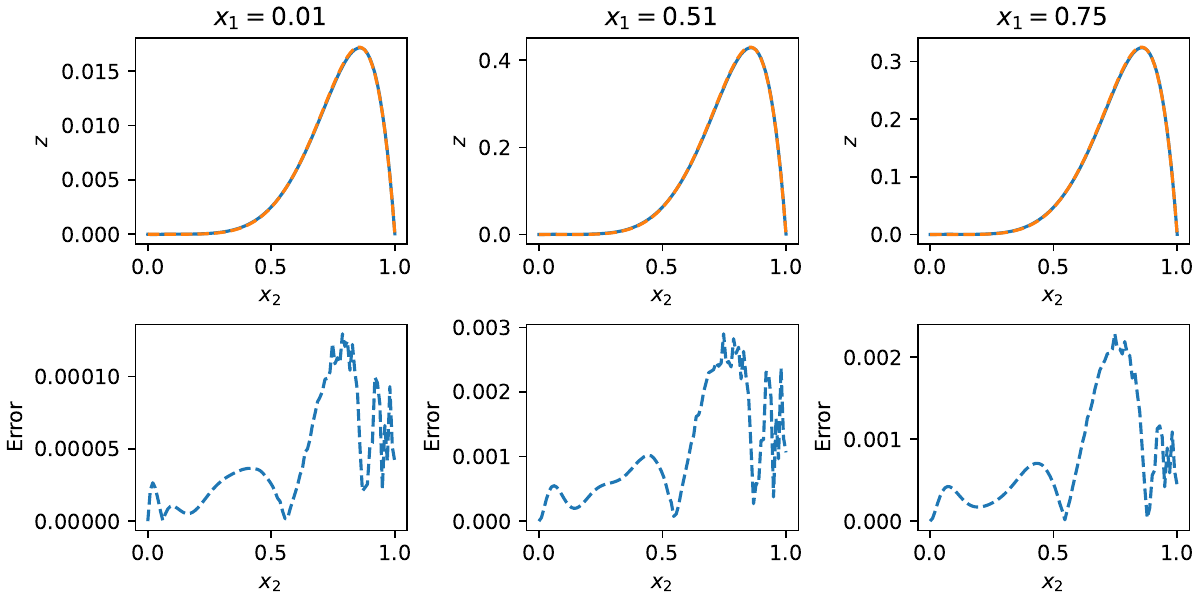}
     \caption{Approximated solution (orange) obtained at various values of $x_1$ along with the exact solution (blue) in the top. The absolute error at the specified $x_1$ in the bottom. \label{particular_t_heat}}
	 \end{subfigure}
	 \begin{subfigure}[b]{1\textwidth}
     \includegraphics[width=1\textwidth, height=0.3\textwidth]{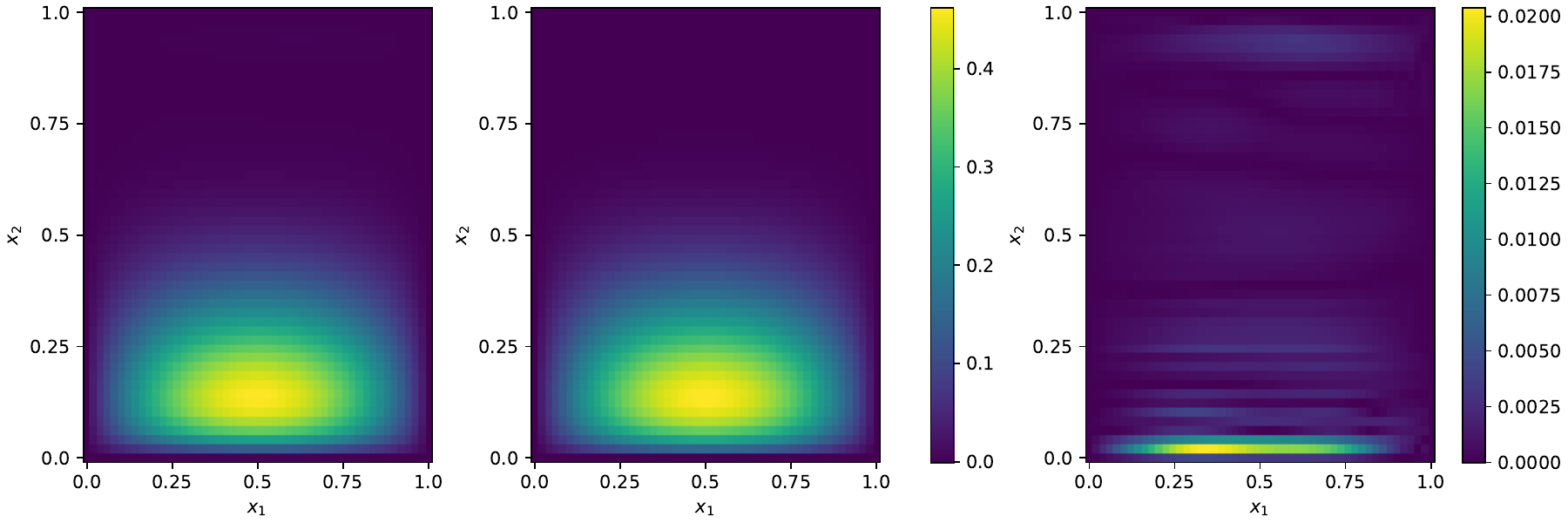}
     \caption{Approximate solution (left), exact solution (center) and absolute error (right) for a particular pair of $m_1$ and $m_2$ during testing for \eqref{Heat equation}.\label{Test_plot_activation_heat}}
	 \end{subfigure}
	 \caption{Numerical results obtained for the parametric heat equation given in equation \eqref{Heat equation} for $X=1$, $T=1$, and $C=5$.\label{Heat equation-example-plots-1}}
\end{figure}

As a demonstration for long-range, we consider the values of $X=5$ and $T=1$. The proposed method is used to solve the system in this domain, and the results obtained are tabulated in Figure \ref{Heat equation-extended-example-plots-2} and \ref{Heat equation-extended-example-plots-1}. We used the Adam optimizer for the training. The results in Figure \ref{L_2_hat_Omega_L_Heat_Extended} and \ref{L_inf_hat_Omega_L_Heat_Extended} demonstrate that increasing the number of neurons improves the accuracy of the method. This validates the Theorem \ref{Convergence-final-theorem}. Further, the loss value graph as the number of epochs increases is plotted in Figure \ref{Loss_plot_different_activation_heat_extended_a_constant}. This further validates that as the neuron increases, the training loss minimizes.

\begin{figure}[H]
    \begin{subfigure}[b]{1\textwidth}
         \includegraphics[width=1\textwidth, height=0.3\textwidth]{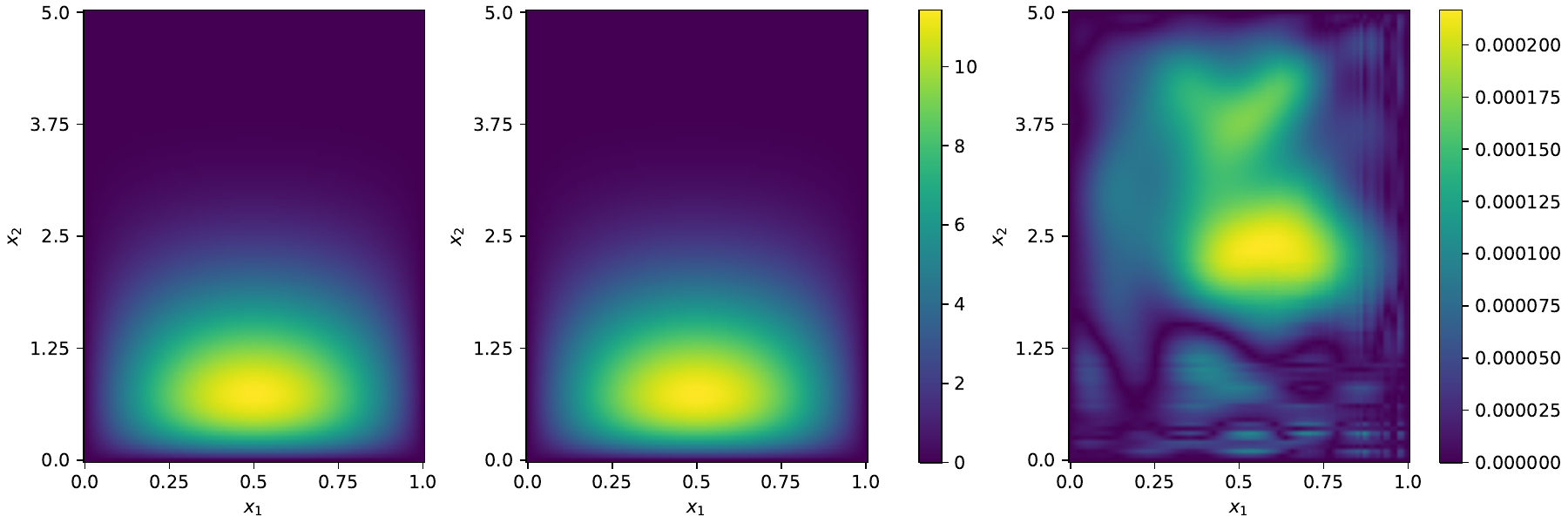}
     \caption{Approximate solution (left), exact solution (center) and absolute error (right) for a particular pair of $m_1$ and $m_2$ during testing for \eqref{Heat equation}.\label{Test_plot_activation_heat_extended}}
    \end{subfigure}
     \caption{Numerical results obtained for the parametric heat equation given in equation \eqref{Heat equation} for $X=5$, $T=1$, and $C=5$.\label{Heat equation-extended-example-plots-2}}
\end{figure}
\begin{figure}[H]
    \begin{subfigure}[b]{0.45\textwidth}
     \includegraphics[width=1\textwidth, height=0.7\textwidth]{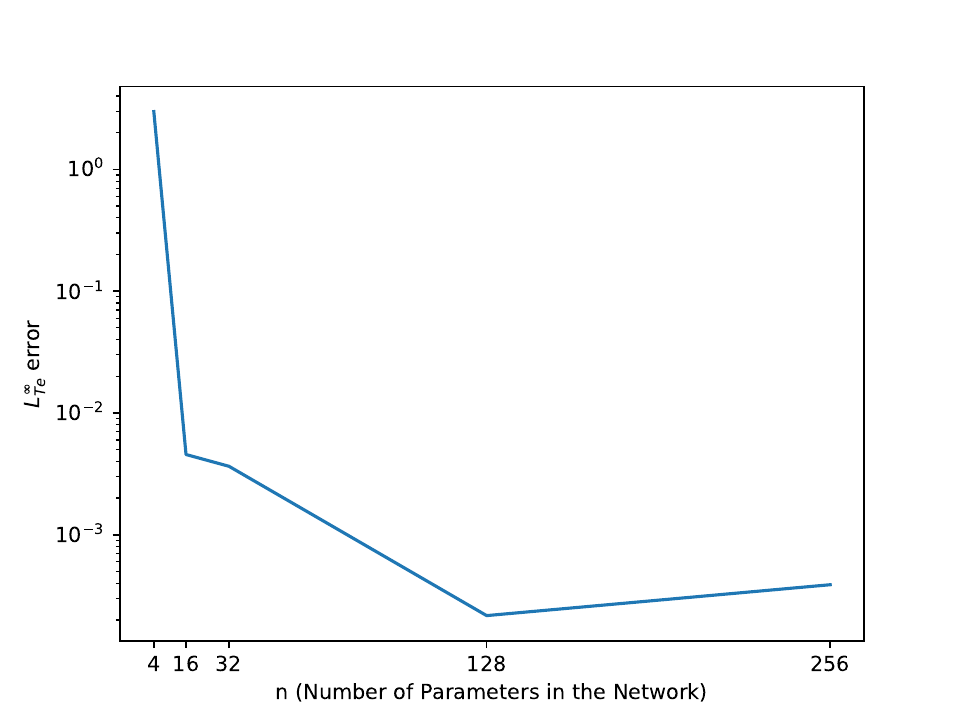}
     \caption{Convergence of $L^{\infty}_{Te}$ error as number of parameters in the neural network $n$ increases. \label{L_inf_hat_Omega_L_Heat_Extended}}
	 \end{subfigure}
	 \begin{subfigure}[b]{0.45\textwidth}
     \includegraphics[width=1\textwidth, height=0.7\textwidth]{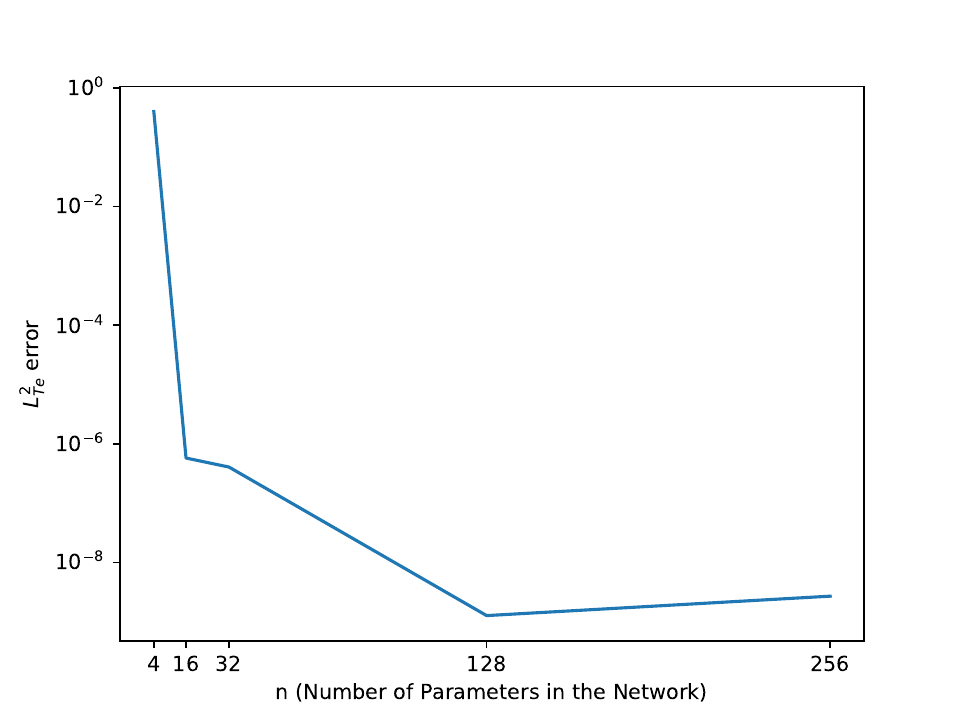}
     \caption{Convergence of $L^{2}_{Te}$ error as the number of parameters in the neural network $n$ increases. \label{L_2_hat_Omega_L_Heat_Extended}}
	 \end{subfigure}
      \begin{subfigure}[b]{0.5\textwidth}
     \includegraphics[width=1\textwidth, height=0.7\textwidth]{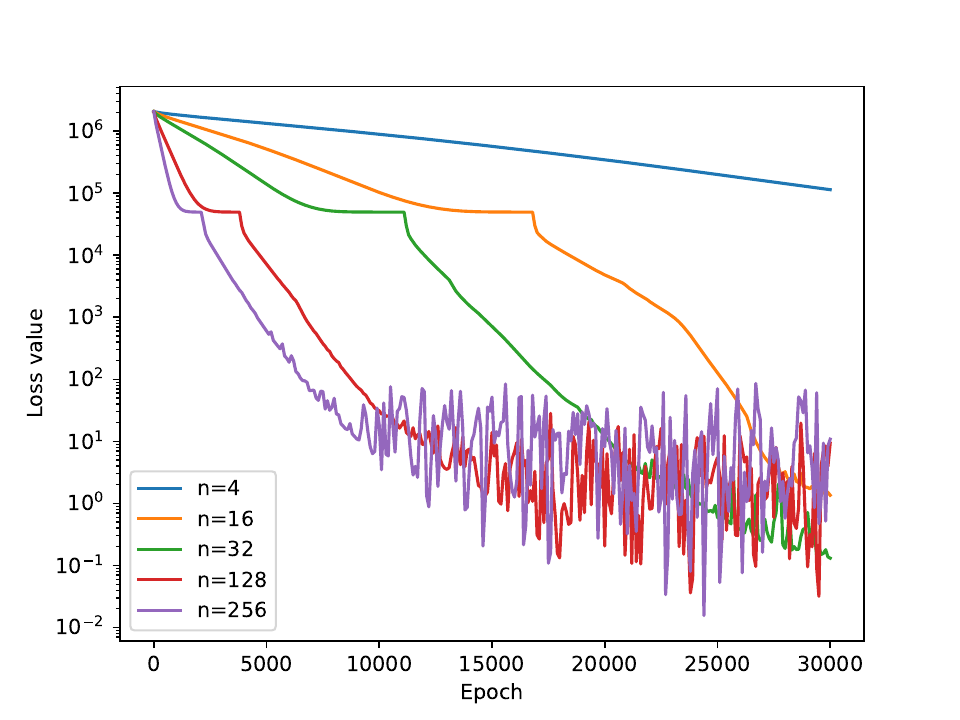}
     \caption{Loss value as the number of epochs increases. (Plotted for every $100^{th}$ epoch) \label{Loss_plot_different_activation_heat_extended_a_constant}}
     \end{subfigure}

	 \caption{Numerical results obtained for the parametric heat equation given in equation \eqref{Heat equation} for $X=5$, $T=1$, and $C=5$.\label{Heat equation-extended-example-plots-1}}
\end{figure}

\subsection{Nonlinear Advection-Diffusion Equation}
This section demonstrates the proposed method for the nonlinear advection-diffusion equation with time and space denoted by $x_1$ and $x_2$ respectively. For this, we consider the following parametrized time-FDE given by
\begin{equation}\label{Nonlinear Advection-Diffusion Equation}
\centering
\begin{array}{rcl}
_0^CD^{\zeta}_{x_1} z(x_1,x_2) - \hat{v} z_{x_2 x_2} +  \mu z_{x_2} & =& f(x_1,x_2; \Upsilon),~(x_1,x_2) \in \Omega = [0,T] \times[0,X],\\
z(x_1,x_2) &=& 0,~(x_1,x_2)\in \partial \Omega,
\end{array}
\end{equation}
where
\begin{equation}\nonumber
\centering
\begin{array}{lcl}
f(x_1,x_2; \Upsilon) &=& \displaystyle C x_1^{m_2} \left(-\frac{x_1^{-\zeta} \Gamma(m_2+1) (T (\zeta-m_2-1)+(m_2+1) x_1) \sin(m_1 x_2 (X-x_2))}{\Gamma(-\zeta+m_2+2)}\right. \\
&&+\mu x_1^{m_2} (x_1-T)^2 \sin^2(m_1 x_1 (X-x_2))+m_1 v (T-x_1) \left(m_1 (X-2 x_2)^2\right. \\
&&\left.\sin(m_1 x_2 (X-x_2))+2 \cos (m_1 x_2 (X-x_2))\right)\bigg{)}.
\end{array}
\end{equation}
The exact solution of this FDE is $z(x_1,x_2) = C(T - x_1) x_1^{m_2} \sin(m_1 (X - x_2) x_2)$. We chose the values of $\hat{v}=1.0$ and $\mu = 0.1$. Further, we take the values of $m_1$ and $m_2$ from $U[1,1.5]$ and set $X=T=1$ and $C=20$. The fractional order is taken to be $0.7$. The numerical integration in the domain of $x_1$ and $x_2$ is performed using the Gauss-Legendre quadrature of degree $20$. The value of $L$ is taken to be $10$, $100$ and $1000$. The value of $n$ (number of parameters in each layer) is $8$, $16$, $32$, $64$, $128$ and $256$. For this problem, we trained the network for $5000$ epochs. We used the Adam optimizer for the first $3000$ epoch, and the rest were performed using the LBFGS optimizer. The error obtained upon varying the number of neurons and the sampling points is presented in Table \ref{Nonlinear Advection-Diffusion Equation-error-table}. Further, we also present the error obtained using various activation functions in Table \ref{Nonlinear Advection-Diffusion Equation-Error-vs-activation-fn}. This table shows that the $\tanh$ activation function performs better than the other functions. Further, it was also noticed that the Relu function performs the worst since it is not a bounded activation function. Also, we present the graphical validation for the approaches in Figure \ref{Nonlinear Advection-Diffusion Equation-example-plots}. The reduction of loss as the number of epochs increases is presented in Figure \ref{Loss_plot_different_activation_advection_diffusion}. The approximate solution obtained at different values of $x_1$ is plotted in Figure \ref{particular_t_advection_convection}. This infers that the error obtained is of order $10^{-2}$. We also present the results obtained during training in Figure \ref{Train_plot_advection_diffusion} and during testing in Figure \ref{Test_plot_advection_diffusion}.

\begin{table}[H]
\centering
\begin{tabular}{|c|c|ccc|}
\hline
\multirow{2}{*}{$n$}&$L$&\multirow{2}{*}{10}&\multirow{2}{*}{100}&\multirow{2}{*}{1000}\\
\cline{2-2}
&Error&&&\\
\hline
\multirow{2}{*}{8}
&$L^2_{Te}$&$5.49e-03$&$7.65e-05$&$1.68e-04$\\
&$L^{\infty}_{Te}$&$1.61e+00$&$9.63e-02$&$2.16e-01$\\
\hline
\multirow{2}{*}{16}
&$L^2_{Te}$&$6.12e-03$&$8.78e-05$&$8.60e-05$\\
&$L^{\infty}_{Te}$&$8.71e-01$&$1.17e-01$&$1.34e-01$\\
\hline
\multirow{2}{*}{32}
&$L^2_{Te}$&$5.58e-03$&$8.32e-05$&$2.32e-04$\\
&$L^{\infty}_{Te}$&$7.96e-01$&$1.00e-01$&$2.11e-01$\\
\hline
\multirow{2}{*}{64}
&$L^2_{Te}$&$3.00e-03$&$6.87e-04$&$1.13e-03$\\
&$L^{\infty}_{Te}$&$5.79e-01$&$2.56e-01$&$3.57e-01$\\

\hline
\multirow{2}{*}{128}
&$L^2_{Te}$&$5.69e-03$&$1.70e-03$&$7.56e-04$\\
&$L^{\infty}_{Te}$&$6.65e-01$&$3.49e-01$&$3.63e-01$\\

\hline
\multirow{2}{*}{256}
&$L^2_{Te}$&$1.69e-03$&$7.52e-04$&$9.58e-04$\\
&$L^{\infty}_{Te}$&$4.06e-01$&$3.42e-01$&$3.28e-01$\\


\hline
\end{tabular}
\caption{$L^{2}_{Te}$ and $L^{\infty}_{Te}$ error obtained as the number of parameters in each layer ($n_p$) and the number of points ($L$) from the domain $\hat{\Omega}$ increases. ($X=T=1$, $C=20$) \label{Nonlinear Advection-Diffusion Equation-error-table}.}
\end{table}
\begin{figure}[H]
     \begin{subfigure}[b]{1\textwidth}
     \includegraphics[width=1\textwidth, height=0.6\textwidth]{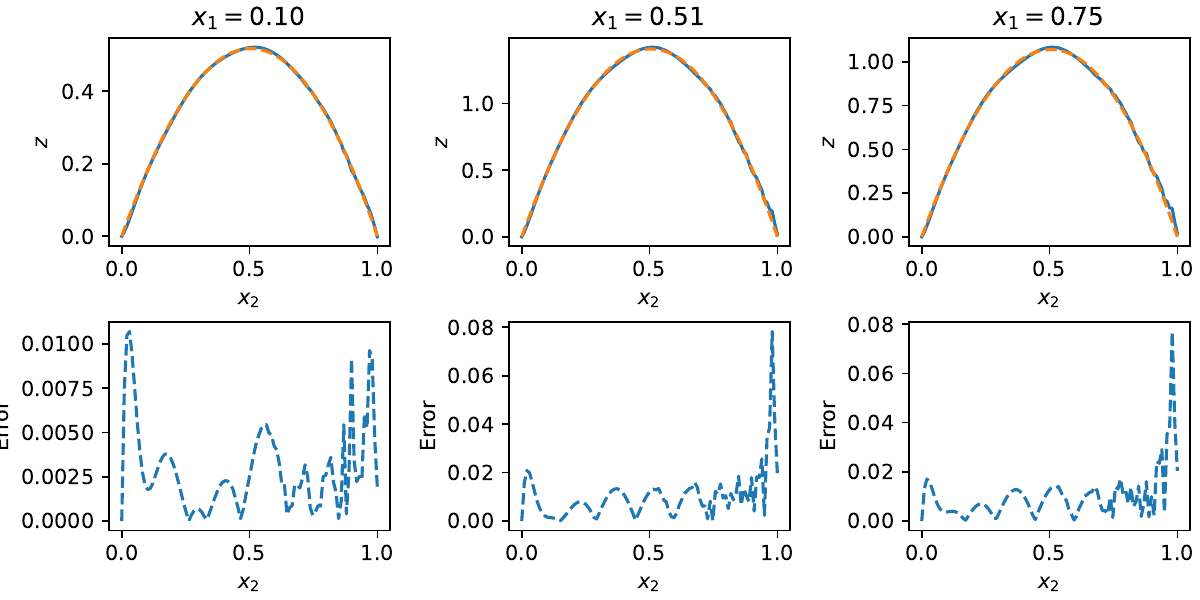}
     \caption{Approximated solution (orange) obtained at various values of $x_1$ along with the exact solution (blue) in the top. The absolute error at the specified $x_1$ in the bottom.  \label{particular_t_advection_convection}}
	 \end{subfigure} 
	 \begin{subfigure}[b]{1\textwidth}
     \includegraphics[width=1\textwidth, height=0.3\textwidth]{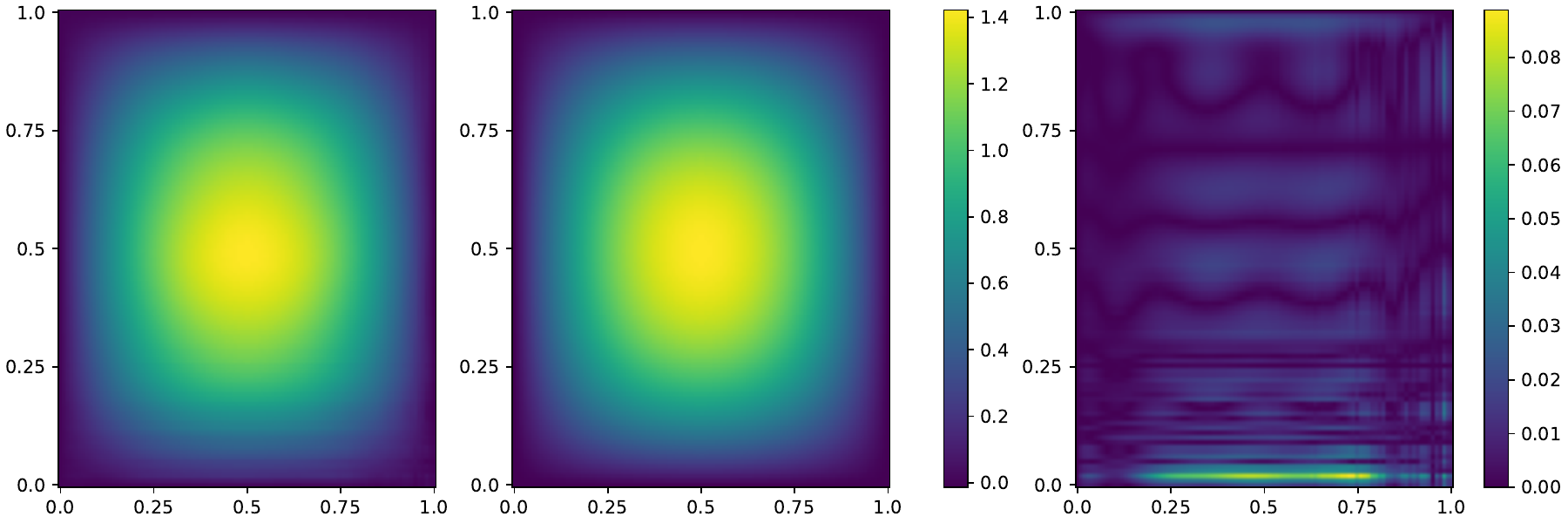}
     \caption{Approximate solution (left), exact solution (center) and absolute error (right) for a particular pair of $m_1$ and $m_2$ during testing. \label{Test_plot_advection_diffusion}}
	 \end{subfigure}
     \begin{subfigure}[b]{0.5\textwidth}
     \includegraphics[width=1\textwidth, height=0.7\textwidth]{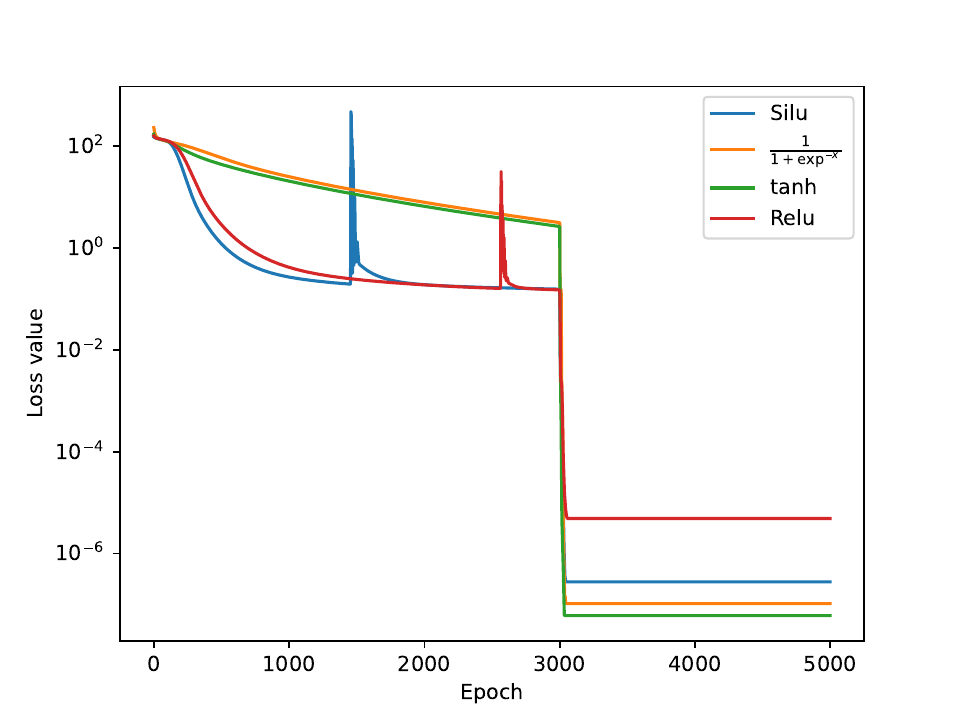}
     \caption{Loss value as the number of epoch increases. \label{Loss_plot_different_activation_advection_diffusion}}
	 \end{subfigure}
	 \caption{Numerical results obtained for the parametric nonlinear advection-diffusion problem in equation \eqref{Nonlinear Advection-Diffusion Equation} at $X=1$, $T=1$, and $C=20$.\label{Nonlinear Advection-Diffusion Equation-example-plots-1}}
\end{figure}

\begin{figure}[H]
	 \begin{subfigure}[b]{1\textwidth}
     \includegraphics[width=1\textwidth, height=0.3\textwidth]{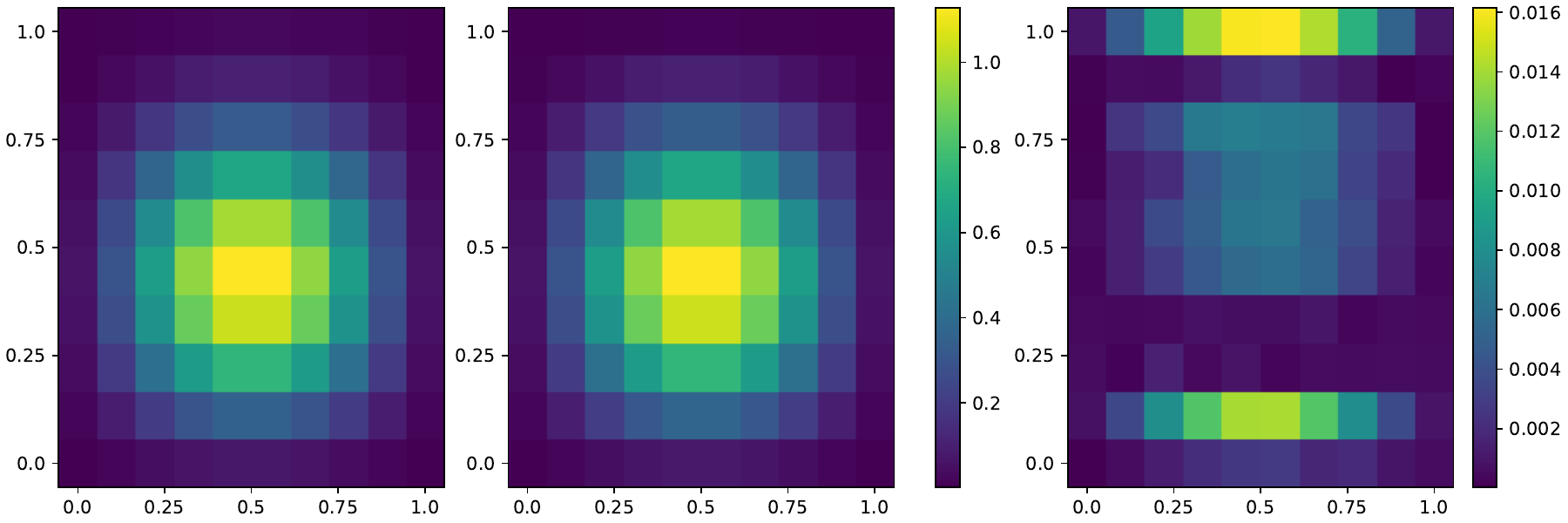}
     \caption{Approximate solution (left), exact solution (center) and absolute error (right) for a particular pair of $m_1$ and $m_2$ during training. \label{Train_plot_advection_diffusion}}
	 \end{subfigure}
	 \caption{Numerical results obtained for the parametric nonlinear advection-diffusion problem in equation \eqref{Nonlinear Advection-Diffusion Equation} at $X=1$, $T=1$, and $C=20$.\label{Nonlinear Advection-Diffusion Equation-example-plots}}
\end{figure}

\begin{table}[H]
\centering
\begin{tabular}{lcccc}
\hline
$\sigma$&$\tanh$&$\displaystyle\frac{1}{1+\exp^{-x}}$&Relu&Silu\\
\hline
$L^2_{Te}$&$8.78e-05$&$1.15e-03$&$5.37e-03$&$1.69e-03$\\
$L_2^{\infty}$&$1.17e-01$&$3.58e-01$&$1.15e+00$&$7.76e-01$\\
\hline
\end{tabular}
\caption{Error obtained on changing the activation function used in the network for equation \eqref{Nonlinear Advection-Diffusion Equation}. ($X=1$, $T=1$, and $C=20$.)\label{Nonlinear Advection-Diffusion Equation-Error-vs-activation-fn}}
\end{table}

Further, to validate the proposed method for longe range domains, we consider the values of $T=10$ and $X = 4$ and apply the proposed method in the domain. The results obtained by using the proposed scheme for $x_1\in[0,10]$ and $x_2\in[0,4]$ are demonstrated in Figure \ref{Nonlinear Advection-Diffusion Equation-Extended-example-plots-1} and \ref{Nonlinear Advection-Diffusion Equation-Extended-example-plots}. The exact, approximated solution and the absolute error obtained are plotted in Figure \ref{Test_plot_activation_advection_diffusion_extended}. Further, the two-dimensional surface plot is presented to show the approximation. This verifies that the proposed method is capable of performing for larger domains under highly nonlinear solutions. Further, increasing the number of neurons produces a better approximation. This is validated by the Figure \ref{L_2_hat_Omega_L_Advection_Diffusion_Extended}, \ref{L_inf_hat_Omega_L_Advection_Diffusion_Extended} and \ref{Loss_plot_different_activation_advection_diffusion}.

\begin{figure}[H]
	 \begin{subfigure}[b]{1\textwidth}
     \includegraphics[width=1\textwidth, height=0.3\textwidth]{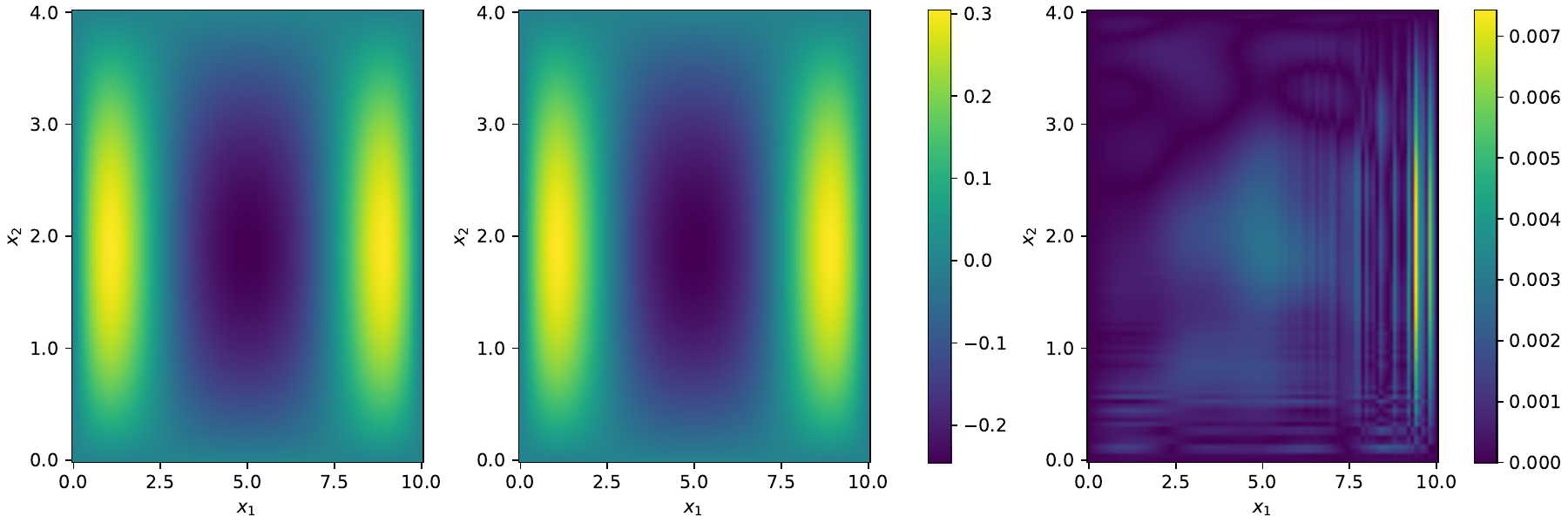}
     \caption{Approximate solution (left), exact solution (center), and absolute error (right) for a particular pair of $m_1$ and $m_2$ during testing. \label{Test_plot_activation_advection_diffusion_extended}}
	 \end{subfigure}
     \caption{Numerical results obtained for the parametric nonlinear advection-diffusion problem in equation \eqref{Nonlinear Advection-Diffusion Equation} at $X=4$, $T=10$, and $C=0.01$.\label{Nonlinear Advection-Diffusion Equation-Extended-example-plots-1}}
\end{figure}
\begin{figure}[H]
\begin{subfigure}[b]{0.45\textwidth}
     \includegraphics[width=1\textwidth, height=0.7\textwidth]{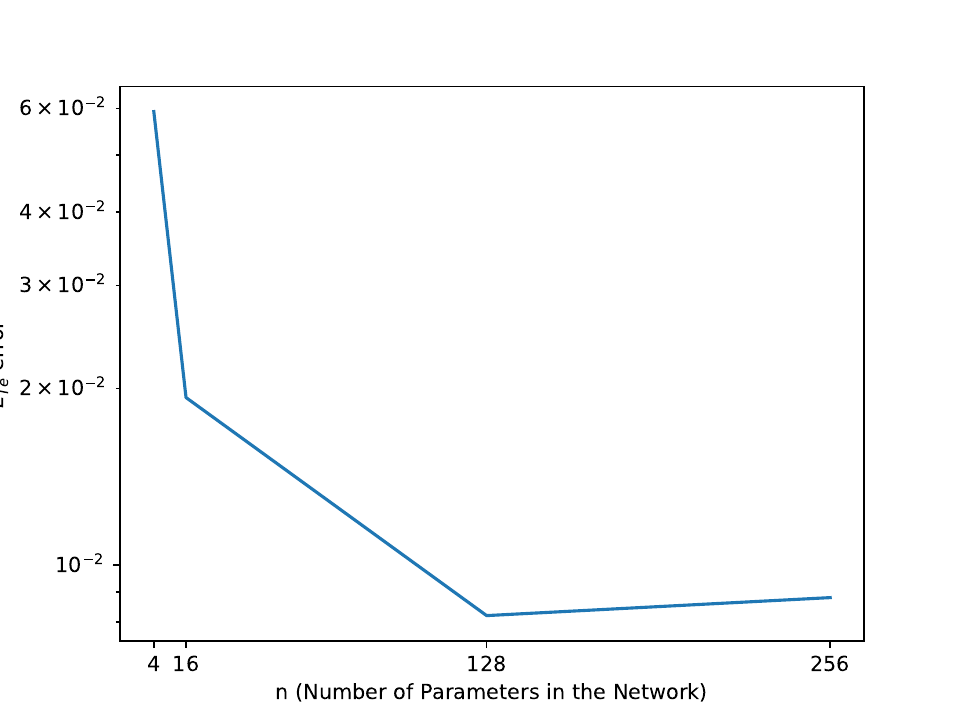}
     \caption{Convergence of $L^{\infty}_{Te}$ error as number of parameters in the neural network $n$ increases. \label{L_inf_hat_Omega_L_Advection_Diffusion_Extended}}
	 \end{subfigure}
	 \begin{subfigure}[b]{0.45\textwidth}
     \includegraphics[width=1\textwidth, height=0.7\textwidth]{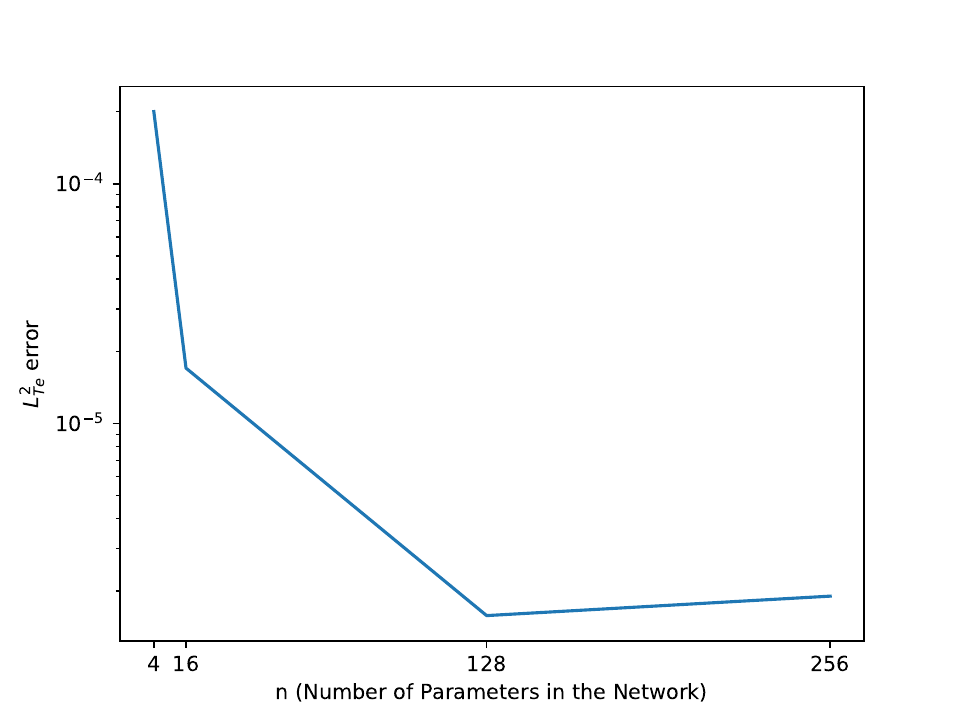}
     \caption{Convergence of $L^{2}_{Te}$ error as the number of parameters in the neural network $n$ increases. \label{L_2_hat_Omega_L_Advection_Diffusion_Extended}}
	 \end{subfigure}
	 \begin{subfigure}[b]{1\textwidth}
     \includegraphics[width=1\textwidth, height=0.3\textwidth]{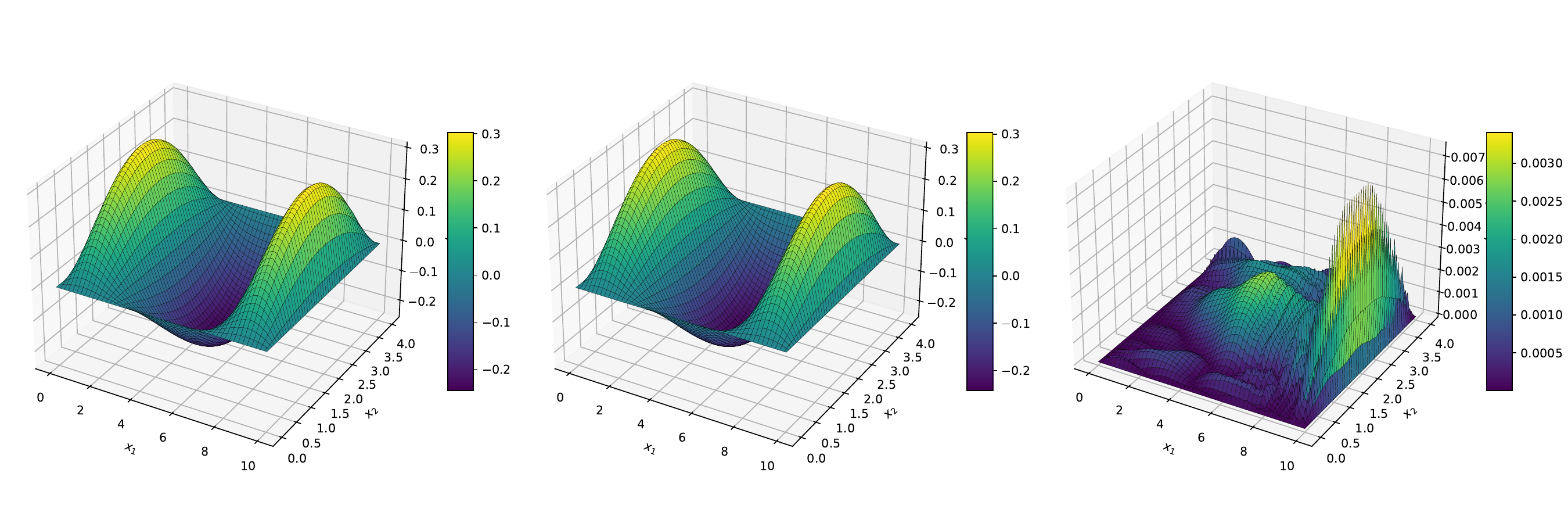}
     \caption{Approximate solution (left), exact solution (center), and absolute error (right) for a particular pair of $m_1$ and $m_2$ during testing in $2d$ view. \label{Test_plot_activation_advection_diffusion_extended_surface}}
	 \end{subfigure}
     \begin{subfigure}[b]{0.5\textwidth}
     \includegraphics[width=1\textwidth, height=0.7\textwidth]{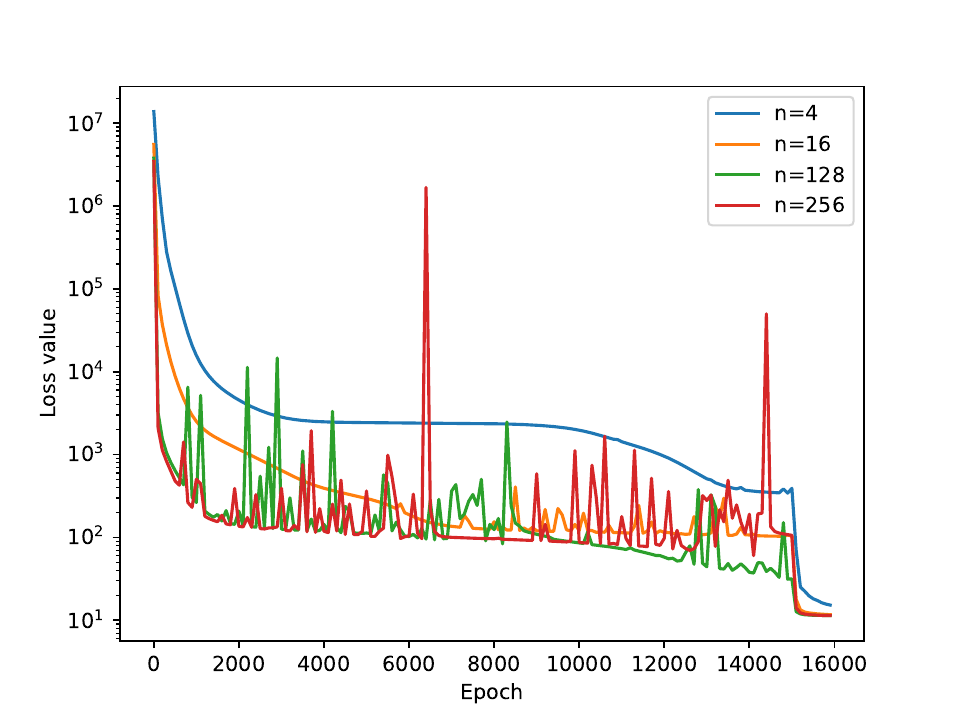}
     \caption{Loss value as the number of epochs increases. \label{Loss_plot_different_activation_Advection_Diffusion_extended_a_constant}}
	 \end{subfigure}
	 \caption{Numerical results obtained for the parametric nonlinear advection-diffusion problem in equation \eqref{Nonlinear Advection-Diffusion Equation} at $X=4$, $T=10$, and $C=0.01$.\label{Nonlinear Advection-Diffusion Equation-Extended-example-plots}}
\end{figure}

\subsection{Non-homogeneous Advection equation}
This subsection presents the application of the proposed method to the non-homogeneous advection equation. For this, we consider the FDE to be parametrized by constant and variable coefficients. 
\subsubsection{With constant coefficents}
This section presents the application of the proposed method to the parametrized time-fractional Advection equation with time and space denoted by $x_1$ and $x_2$ respectively. The equation considered for this study is as follows:
\begin{equation}\label{Advection equation With constant coefficents}
\centering
\begin{array}{lcl}
_0^CD^{\zeta}_{x_1}z(x_1,x_2) - a z_{x_2} &=& \displaystyle f(x_1,x_2; \Upsilon),~(x_1,x_2) \in [0,1]\times[0,1],
\end{array}
\end{equation}
where the function $f$ is chosen to be 
\begin{equation}\nonumber
\centering
\begin{array}{lcl}
f(x_1,x_2; \Upsilon)&=&\displaystyle 20^a x_1^{m_1} \left(\frac{x_1^{-\zeta} \Gamma (m_1+1) (\zeta+(m_1+1) (x_1-1)) \sin (a (x_2-1) x_2)}{\Gamma (-\zeta+m_1+2)}-a \zeta (x_1-1) \right.\\
&&\times(2 x_2-1) \cos (a (x_2-1) x_2)\bigg{)}.
\end{array}
\end{equation}
The exact solution of this system is given by 
\begin{equation}\nonumber
\centering
u(x_1,x_2) = 20^a (1-x_1) x_1^{m_1} \sin (a (1-x_2) x_2).
\end{equation}
It could be seen from the above equation \eqref{Advection equation With constant coefficents}, the parameter $a$ and the parameter $m_1$ are randomly generated. We chose it to follow $U[1,1.5]$. The results obtained by using the proposed method are presented in Figure \ref{Advection equation With constant coefficents-example-plots}. The error obtained by increasing the number of parameters in each layer is given in Figure \ref{L_2_hat_Omega_L_advection} and \ref{L_inf_hat_Omega_L_advection} by using the $L^2_{Te}$ and $L^{\infty}_{Te}$ respectively. From these plots, it can be inferred that the error in approximating the unknown function $z$ is reduced as the number of neurons increases from $4$ to $64$. Further, it could be seen from Figure \ref{Loss_plot_different_activation_advection_a_constant} that as the number of neurons is $4$ and $8$, the loss function does not get reduced. However, as the number of neurons increases, the loss is reduced. Further, we also present the results obtained at different values of $x_1$ in Figure \ref{particular_t_advection_a_constant}. This shows that the error from the proposed approach to varying values of $x_2$ is of order $10^{-2}$. Further, the results obtained on training are presented in Figure \ref{Train_plot_activation_advection_a_constant} along with the exact solution and the error plot. This verifies that the error is of order $10^{-2}$. The network was trained for a total of $500$ epochs by using the LBFGS optimizer.
\begin{figure}[H]
     \begin{subfigure}[b]{0.5\textwidth}
     \includegraphics[width=1\textwidth, height=0.7\textwidth]{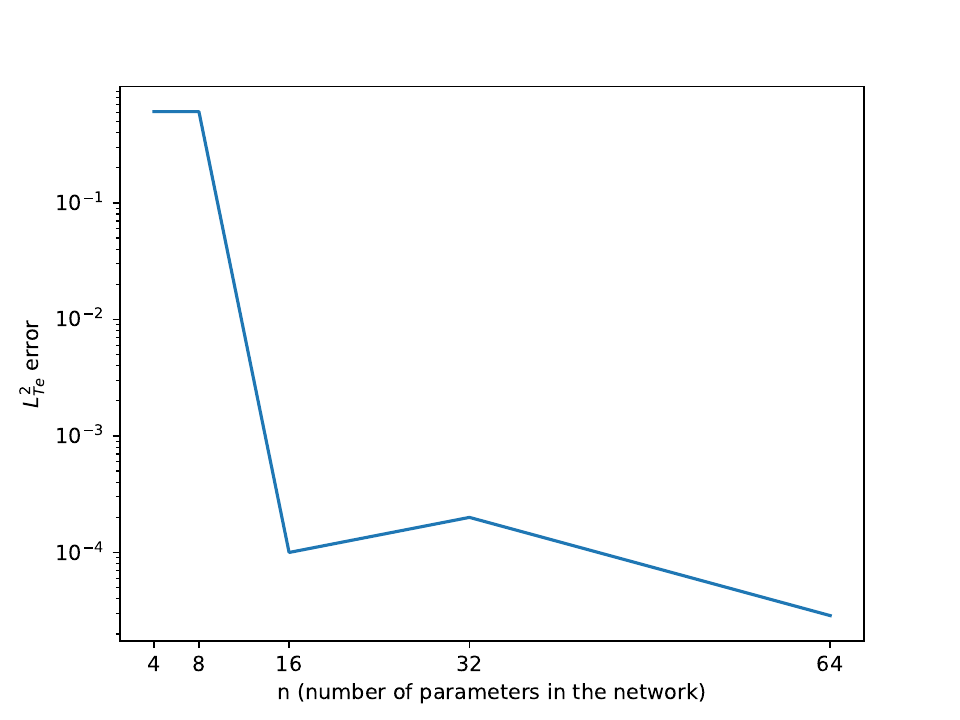}
     \caption{Convergence of $L^{2}_{Te}$ error as number of neurons $n$ increases. \label{L_2_hat_Omega_L_advection}}
	 \end{subfigure}
	 \begin{subfigure}[b]{0.5\textwidth}
     \includegraphics[width=1\textwidth, height=0.7\textwidth]{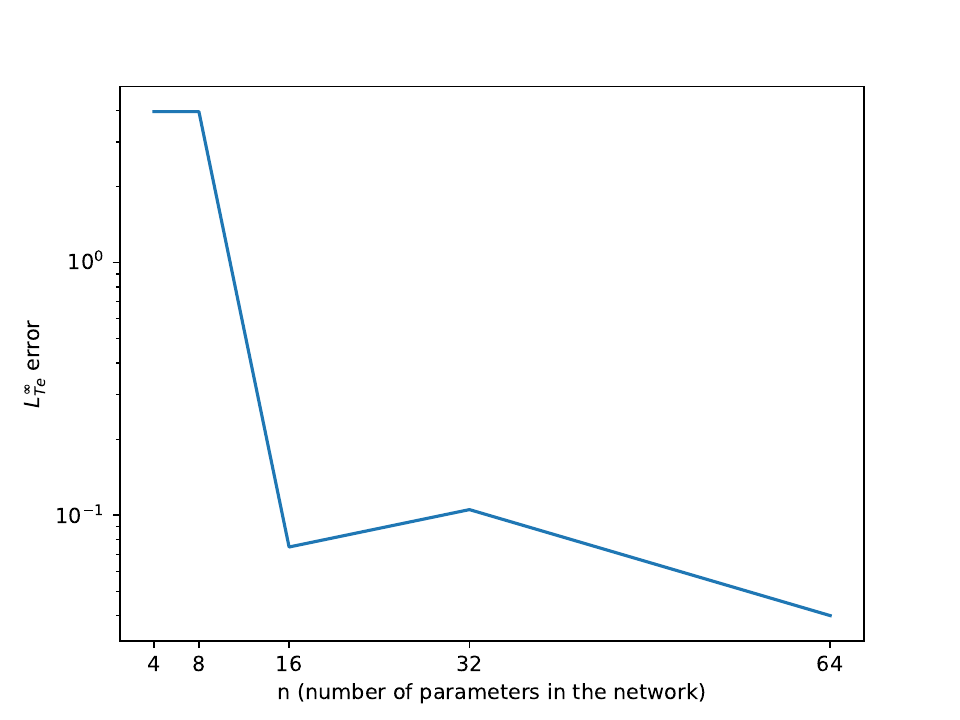}
     \caption{Convergence of $L^{\infty}_{Te}$ error as number of neurons $n$ increases. \label{L_inf_hat_Omega_L_advection}}
	 \end{subfigure}
	 \caption{Numerical results obtained for the parametric advection equation with constant coefficient given in equation \eqref{Advection equation With constant coefficents}.\label{Advection equation With constant coefficents-example-plots}}
\end{figure}

\begin{figure}[H]
 \begin{subfigure}[b]{1\textwidth}
     \includegraphics[width=1\textwidth, height=0.3\textwidth]{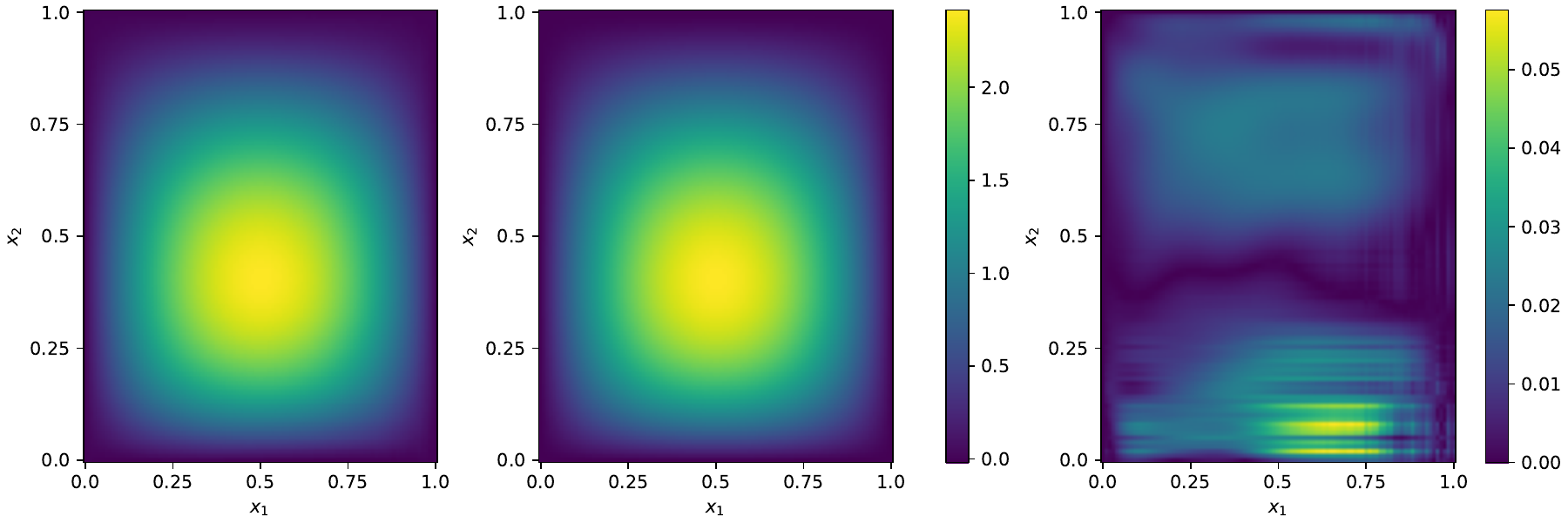}
     \caption{Approximate solution (left), exact solution (center) and absolute error (right) for a particular pair of $m_1$ and $m_2$ during training. \label{Train_plot_activation_advection_a_constant}}
	 \end{subfigure}
 \begin{subfigure}[b]{1\textwidth}
     \includegraphics[width=1\textwidth, height=0.6\textwidth]{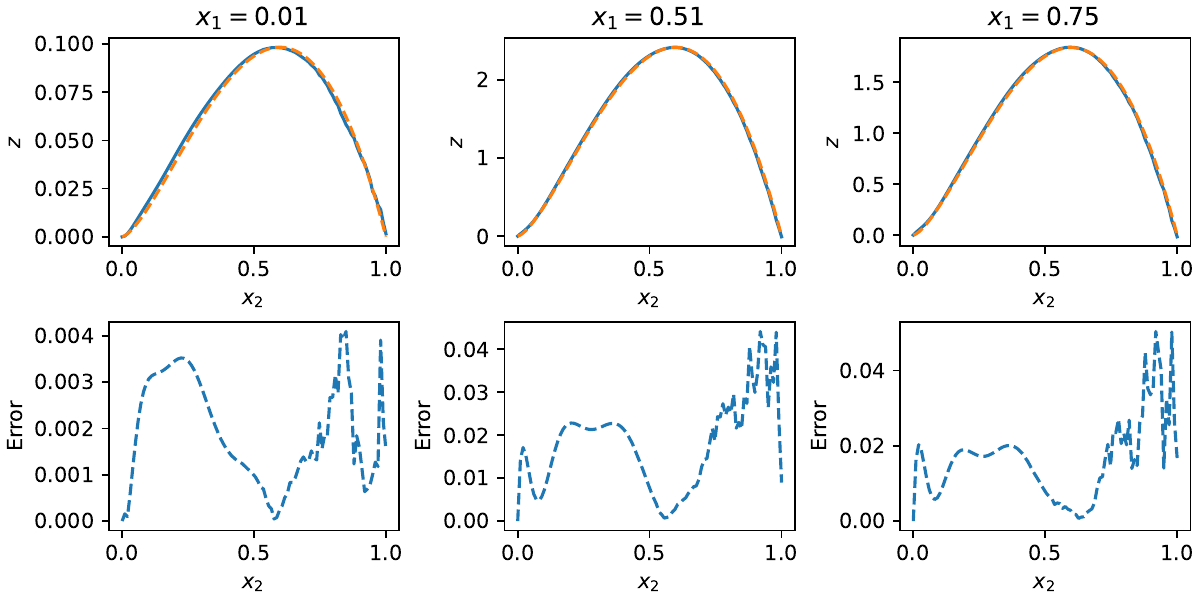}
     \caption{Approximated solution (orange) obtained at various values of $x_1$ along with the exact solution (blue) in the top. The absolute error at the specified $x_1$ in the bottom. \label{particular_t_advection_a_constant}}
	 \end{subfigure}
      \begin{subfigure}[b]{0.5\textwidth}
     \includegraphics[width=1\textwidth, height=0.7\textwidth]{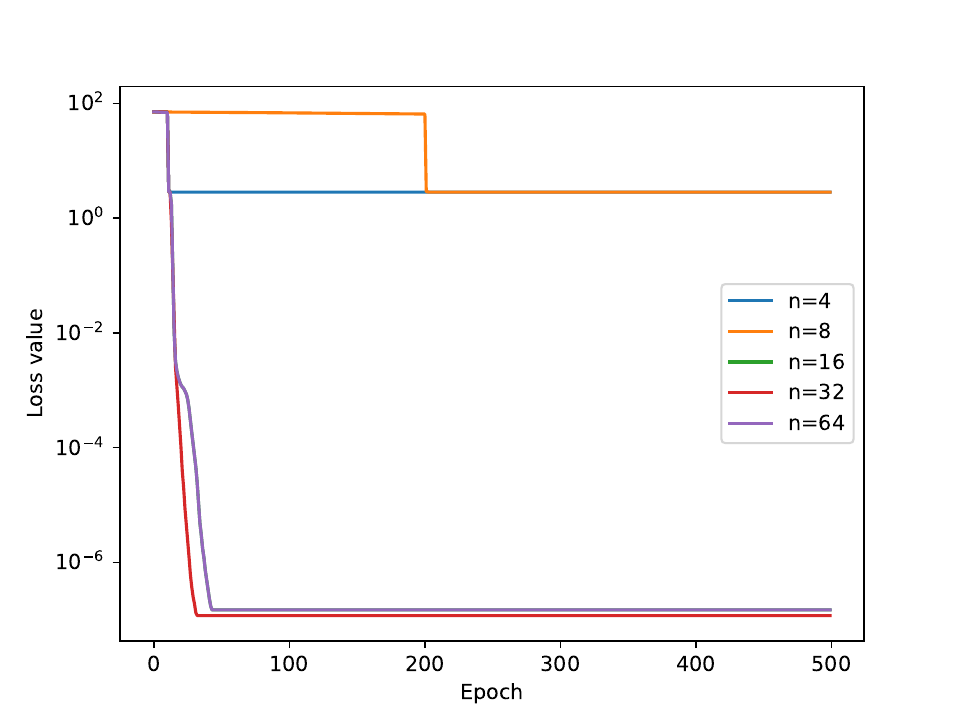}
     \caption{Loss values as the number of epochs increase. \label{Loss_plot_different_activation_advection_a_constant}}
	 \end{subfigure}
	 \caption{Numerical results obtained for the parametric advection equation with constant coefficient given in equation \eqref{Advection equation With constant coefficents}.\label{Advection equation With constant coefficents-example-plots-1}}
\end{figure}

\subsubsection{With variable coefficents}
This section presents the application of the proposed method to the parametrized time-fractional Advection equation with an external forcing term, where the coefficient is a function that is randomly chosen from the Gaussian random field. Further time and space are denoted by $x_1$ and $x_2$ respectively. The equation considered for this study is as follows:
\begin{equation}\label{Advection equation With variable coefficents}
\centering
\begin{array}{lcl}
_0^CD^{\zeta}_{x_1}z(x_1,x_2) - a(x_1) z_{x_2} &=& \displaystyle f(x_1,x_2; \Upsilon),~(x_1,x_2) \in [0,1]\times[0,1],
\end{array}
\end{equation}
where the function $f$ is chosen to be 
\begin{equation}\nonumber
\centering
\begin{array}{lcl}
f(x_1,x_2; \Upsilon)&=&\displaystyle \frac{2 (x_2-1) x_2 x_1^{2-\zeta} (\zeta+3 x_1-3) \sin (x_2)}{\Gamma (4-\zeta)}\\
&&\displaystyle-a(x_1) (x_1-1) x^3 ((2 x_2-1) \sin (x_2)+(x_2-1) x_2 \cos (x_2)).
\end{array}
\end{equation}
The exact solution of this system is given by 
\begin{equation}\nonumber
\centering
u(x_1,x_2) = 200(1 - x_1)(x_1^2)(1-x_2)x_2\sin(x_2).
\end{equation}
In this example, we consider the function $a(x_1)$ the randomly initialized function. To obtain training functions, we generate the function $a(x_1)$ from the Gaussian Random Fields following $\displaystyle N\left(0,\frac{1}{(N+1)^2}\right)$. The fractional order is fixed to be $0.7$. The network is trained for a total of $500$ epochs using the LBFGS optimizer. The numerical integration for $x_1$ and $x_2$ is performed using the Gauss-Legendre quadrature rule by fixing the degree to $20$. The graphical results of the proposed approach are presented in Figure \ref{Advection equation With variable coefficents-example-plots}. The loss value as the number of epochs increases is presented in Figure \ref{Loss_plot_different_activation_advection_a_variable}. In addition, we present the solution trajectory of the approximated solution in Figure \ref{particular_t_advection_a_variable} at particular values of $x_1$. This plot shows that the error is of order $10^{-2}$. Also, we present the heatmap of the approximated solution, the exact solution, and the absolute error in Figure \ref{Train_plot_activation_advection_a_variable}.

\begin{figure}[H]
	 \begin{subfigure}[b]{1\textwidth}
     \includegraphics[width=1\textwidth, height=0.3\textwidth]{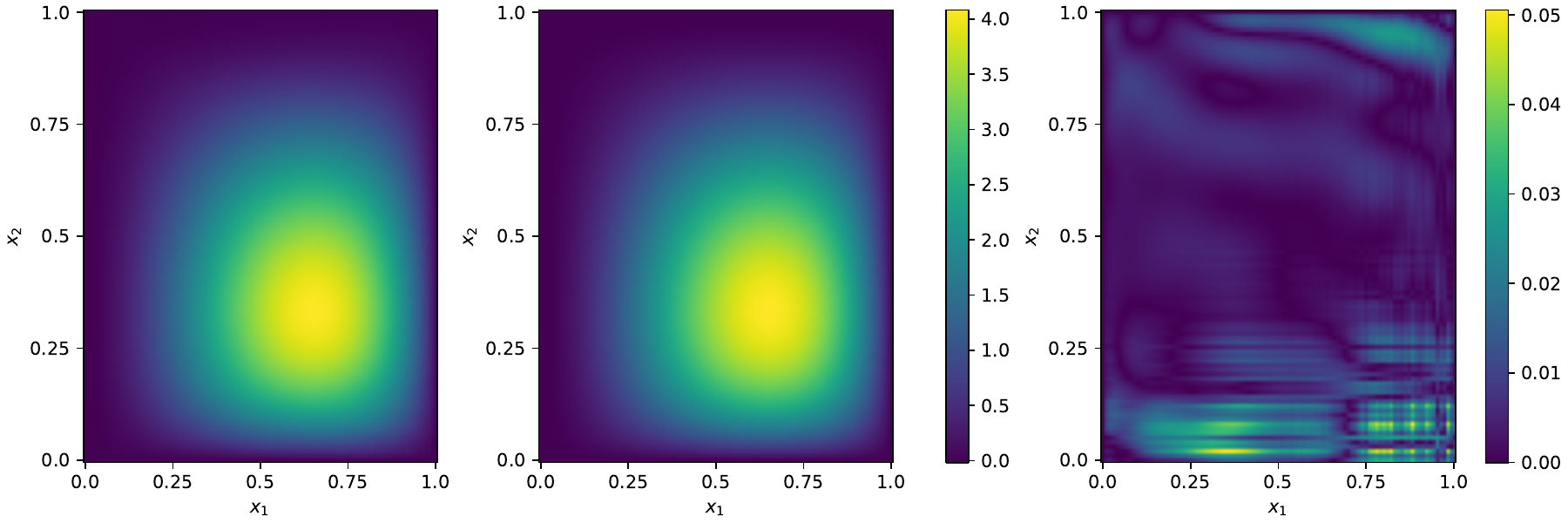}
     \caption{Approximate solution (left), exact solution (center) and absolute error (right) for a particular pair of $m_1$ and $m_2$ during training. \label{Train_plot_activation_advection_a_variable}}
	 \end{subfigure}
	 \caption{Numerical results obtained for the parametric advection equation with variable coefficient given in equation \eqref{Advection equation With variable coefficents}.\label{Advection equation With variable coefficents-example-plots}}
\end{figure}

\begin{figure}[H]
	 \begin{subfigure}[b]{1\textwidth}
     \includegraphics[width=1\textwidth, height=0.6\textwidth]{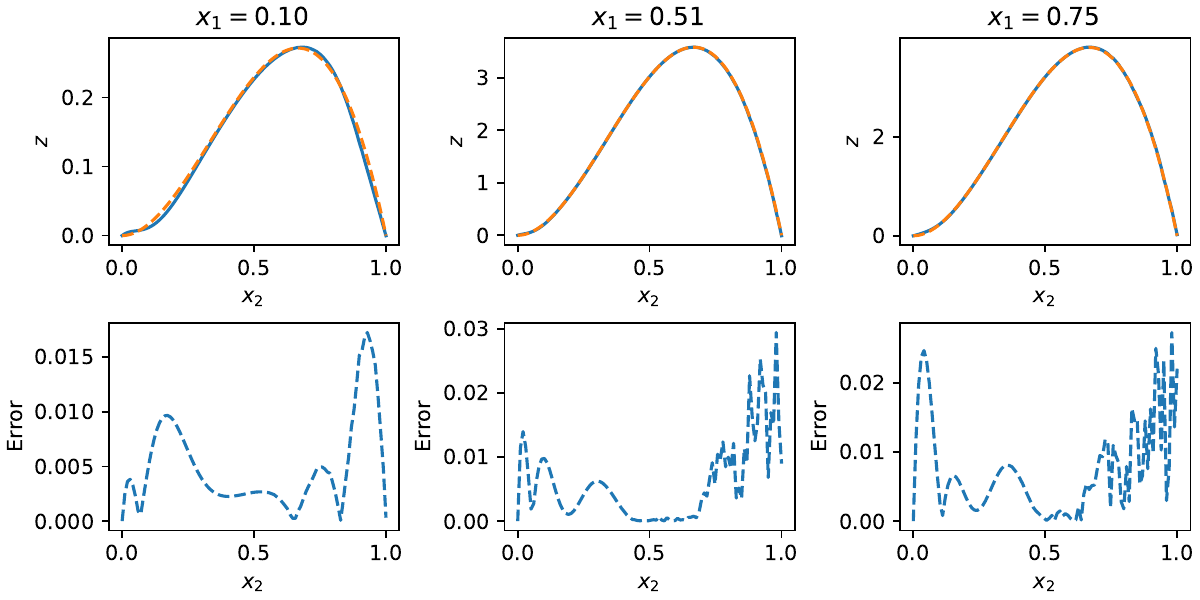}
     \caption{Approximated solution (orange) obtained at various values of $x_1$ along with the exact solution (blue) in the top. The absolute error at the specified $x_1$ in the bottom.. \label{particular_t_advection_a_variable}}
	 \end{subfigure}
     \begin{subfigure}[b]{0.5\textwidth}
     \includegraphics[width=1\textwidth, height=0.7\textwidth]{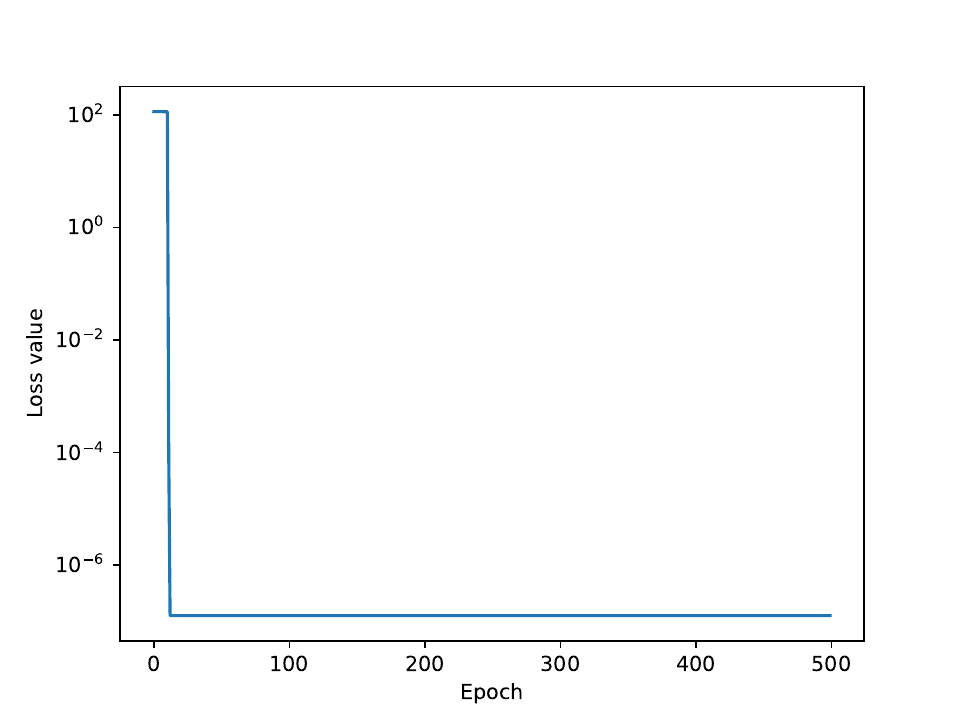}
     \caption{Loss values as the number of epoch increases. \label{Loss_plot_different_activation_advection_a_variable}}
	 \end{subfigure}
	 \caption{Numerical results obtained for the parametric advection equation with variable coefficient given in equation \eqref{Advection equation With variable coefficents}.\label{Advection equation With variable coefficents-example-plots-1}}
\end{figure}


\begin{table}[H]
\centering
\begin{tabular}{ccccccc}
\hline
Problem&$\Upsilon$&Domain&$\zeta$&N&m&\\
\hline
\eqref{linear-1-d-example-eq}&$m_1$, $m_2$&$U[3,5]$&$1.5$&$10$&$10$&$ $\\
\eqref{Heat equation}&$m_1$, $m_2$&$U[5, 7]$&$0.7$&$10$&$10$&$ $\\
\eqref{Nonlinear Advection-Diffusion Equation}&$m_1$, $m_2$&$U[1,1.5]$&$0.7$&$10$&$20$&$ $\\
\eqref{Advection equation With constant coefficents}&$a,~m_1$&$U[1,1.5]$&$0.7$&$10$&$15$&$ $\\
\eqref{Advection equation With variable coefficents}&$a(x_1)$& GRF $\sim$ $\displaystyle N\left(0,\frac{1}{(N+1)^2}\right)$&$0.7$&$10$&$20$&$ $\\
\hline
\end{tabular}
\caption{Details on the randomly chosen parameters $\Upsilon$ for all the problems.}
\end{table}


\section{Conclusion}\label{conclusion}
In this paper, the traditional numerical scheme of the Legendre-Galerkin method and the physics-informed neural network strategy are combined to develop a novel method for simulating the linear and non-linear parametric time-fractional differential equation. It is also noted that there is no numerical method to simulate parametric fractional order equations in the literature. A weak form of the differential equation is used to avoid the need to calculate higher-order derivatives. A clear error analysis is presented theoretically under the assumption of stability, Lipchitzness, and boundedness of the non-linear operator. The overall error is reduced as the number of parameters in the network and the number of sampling points used in the Monte Carlo simulation increase. This theoretical result is also numerically proved by considering different well-known examples parametrized in the source function, the coefficients, and the solution. The results demonstrate the effect of increasing the number of neurons and the sampling points. Additionally, the effect of using different activation function is also shown. From which, it was concluded that for unbounded activation the error is large. This further supports the theory.

One limitation of the proposed approach is that when the boundary condition is periodic, the proposed method is not adaptable and requires test functions other than the Legendre polynomial to approximate the solution. This is a potential area for future research. Additionally methods for studying impulsive and space-time fractional systems could also be explored. Additionally, during simulation it was observed that increasing number of basis or the degree of Legendre polynomial leads to instability in the calculation of the fractional derivative using the series formula. Thus, subsequent research to overcome this could be conducted.  

\section*{Data Availability Statement}
The data used in this research is available/mentioned within the manuscript.
\section*{Funding}
NA
\section*{Conflict of interest}
This work does not have any conflicts of interest. 


\addcontentsline{toc}{section}{References}

\newpage
\section*{Appendix:The matrices $\hat{H}$ and $\hat{M}$ generated for $N=12$.} \label{appendix-1}
\begin{table}[H]
    \centering
  \adjustbox{padding=-30pt, max width=\textwidth, max height=0.5\textheight}{
  \begin{tabular}{cc}
      \rotatebox{270}{
$\hat{H} = \left[
\begin{array}{rrrrrrrrrrrrrr}
 -6.44 & -3.00 & 0.58 & -0.23 & 0.11 & -0.06 & 0.04 & -0.02 & 0.02 & -0.01 & 0.0 & -0.00 \\
 2.80 & -6.23 & -3.98 & 0.91 & -0.39 & 0.21 & -0.13 & 0.08 & -0.06 & 0.04 & -0.03 & 0.02 \\
 -9.14 & 1.38 & -6.33 & -4.81 & 1.22 & -0.57 & 0.32 & -0.20 & 0.14 & -0.10 & 0.075 & -0.05 \\
 5.94 & -7.47 & 0.66 & -6.50 & -5.56 & 1.51 & -0.74 & 0.44 & -0.28 & 0.20 & -0.14 & 0.11 \\
 -12.09 & 3.43 & -6.85 & 0.17 & -6.71& -6.25 & 1.79 & -0.91 & 0.55 & -0.37 & 0.26 & -0.19 \\
 8.98& -9.15 & 2.29 & -6.56 & -0.19 & -6.92 & -6.88 & 2.05 & -1.08 & 0.67 & -0.46 & 0.33 \\
 -15.07 & 5.31 & -7.95 & 1.61 & -6.42 & -0.50 & -7.14 & -7.47 & 2.30 & -1.24 & 0.79 & -0.55 \\
 12.01 & -10.91 & 3.70 & -7.31 & 1.14 & -6.37 & -0.766 & -7.35 & -8.03 & 2.54 & -1.39 & 0.90 \\
 -18.07 & 7.15 & -9.16 & 2.77 & -6.95 & 0.78 & -6.36 & -1.00 & -7.55 & -8.55 & 2.76 & -1.54 \\
 15.02 & -12.69 & 5.04 & -8.21 & 2.15 & -6.72 & 0.49 & -6.38 & -1.21 & -7.76 & -9.06 & 2.98 \\
\end{array}
\right]$}   

&  

\adjustbox{angle=270, padding=0.1pt}{
$\small\hat{M} = \left[
\begin{array}{rrrrrrrrrrrrrr}
 \text{1.9e-17} & -2. & \text{1.8e-16} & -\text{3.3e-16} & -\text{1.2e-15} & -\text{1.3e-14} & \text{1.5e-13} & \text{1.7e-13} & -\text{7.2e-13} & -\text{6.0e-12} & \text{5.4e-11} & -\text{1.6e-10} \\
 -\text{2.4e-15} & -\text{4.4e-16} & -2. & \text{8.8e-16} & \text{1.8e-15} & \text{1.0e-14} & -\text{2.3e-14} & \text{3.0e-14} & \text{6.3e-13} & \text{3.4e-12} & \text{1.5e-11} & \text{1.5e-11} \\
 \text{4.7e-15} & -\text{3.6e-15} & -\text{8.3e-16} & -2. & \text{4.8e-15} & -\text{1.4e-14} & \text{1.1e-13} & -\text{1.6e-13} & \text{3.6e-12} & \text{1.0e-11} & \text{3.6e-12} & \text{3.0e-10} \\
 \text{2.5e-15} & -\text{9.8e-15} & -\text{7.7e-16} & \text{3.3e-15} & -2. & -\text{5.3e-15} & -\text{3.7e-14} & \text{3.4e-14} & \text{1.4e-12} & \text{1.1e-11} & -\text{5.7e-11} & \text{2.9e-11} \\
 -\text{8.4e-14} & -\text{2.1e-14} & \text{1.1e-14} & \text{6.7e-15} & \text{1.2e-14} & -2. & -\text{1.0e-13} & \text{1.5e-13} & \text{9.7e-13} & \text{3.1e-12} & -\text{8.0e-11} & \text{1.6e-10} \\
 \text{1.4e-13} & \text{2.6e-13} & \text{1.2e-13} & \text{1.9e-13} & \text{8.4e-14} & -\text{1.6e-13} & -2. & \text{1.7e-13} & -\text{3.4e-12} & -\text{6.8e-12} & -\text{6.4e-11} & \text{2.8e-10} \\
 \text{4.1e-12} & -\text{5.2e-14} & \text{2.5e-13} & -\text{1.4e-12} & -\text{9.0e-13} & -\text{2.3e-13} & \text{1.4e-13} & -2. & \text{1.6e-12} & -\text{6.9e-12} & -\text{6.8e-11} & -\text{1.6e-9} \\
 \text{2.2e-12} & -\text{8.0e-12} & -\text{1.4e-12} & -\text{3.0e-12} & -\text{6.0e-13} & \text{3.3e-12} & \text{3.5e-12} & \text{3.6e-12} & -2. & -\text{4.6e-12} & -\text{3.7e-11} & \text{3.48e-11} \\
 -\text{1.6e-11} & -\text{9.8e-12} & -\text{1.7e-11} & \text{2.9e-11} & -\text{7.0e-12} & \text{1.9e-12} & -\text{6.6e-12} & \text{3.2e-11} & \text{4.8e-12} & -2. & \text{1.2e-11} & -\text{1.0e-10} \\
 -\text{1.8e-10} & \text{7.4e-11} & -\text{6.4e-11} & -\text{7.4e-11} & \text{2.0e-11} & \text{1.2e-10} & \text{5.1e-11} & -\text{8.2e-11} & -\text{3.1e-11} & \text{2.6e-11} & -2. & -\text{5.3e-10} \\
\end{array}
\right]$}\\
    \end{tabular}}
\end{table}

\end{document}